\documentclass[11pt,letterpaper]{article}
\usepackage{graphicx}
\usepackage{epsfig,subfigure}
\usepackage{amssymb}
\usepackage{amsthm}
\usepackage{amsmath}
\usepackage[margin=1in]{geometry}

\newtheorem{theorem}{Theorem}[section]
\newtheorem{lemma}[theorem]{Lemma}
\newtheorem{corollary}[theorem]{Corollary}
\newtheorem{proposition}[theorem]{Proposition}

\newtheorem{definition}[theorem]{Definition}

\newcommand{\R}{{\mathbb R}}

\newcommand{\Sp}{\mathbb{S}}
\newcommand{\E}{\operatorname{\mathbb{E}}}
\newcommand{\cT}{\mathcal{T}}

\renewcommand{\P}{\operatorname{\mathbb{P}}}
\newcommand{\PT}{{\cal P}_\Delta}
\newcommand{\PTc}{{\cal P}_{\Delta^\perp}}

\newcommand{\ones}{\mathbf 1}

\newcommand{\ba}{\mathbf{a}}
\newcommand{\bx}{\mathbf{x}}
\newcommand{\by}{\mathbf{y}}
\newcommand{\bz}{\mathbf{z}}
\newcommand{\bg}{\mathbf{g}}
\newcommand{\bu}{\mathbf{u}}
\newcommand{\bv}{\mathbf{v}}
\newcommand{\bw}{\mathbf{w}}
\newcommand{\bxs}{\mathbf{x}^\star}
\newcommand{\A}{\mathcal{A}}
\newcommand{\ch}{\mathrm{conv}}
\newcommand{\n}{\mathrm{null}}

\title{The Convex Geometry of Linear Inverse Problems}

\author{Venkat Chandrasekaran$^m$, Benjamin Recht$^w$, Pablo A. Parrilo$^m$, and Alan S. Willsky$^m$ \thanks{Email: \{venkatc,parrilo,willsky\}@mit.edu; brecht@cs.wisc.edu.
This work was supported in part by AFOSR grant FA9550-08-1-0180, in part by a MURI through ARO grant W911NF-06-1-0076, in part by a MURI through AFOSR grant FA9550-06-1-0303,  in part by NSF FRG 0757207, in part through ONR award N00014-11-1-0723, and NSF award CCF-1139953. }
\vspace{0.25in} \\ $^m$ Laboratory for Information and Decision Systems \\ Department of Electrical Engineering and Computer Science \\ Massachusetts Institute of Technology \\ Cambridge, MA 02139 USA  \vspace{0.1in} \\ $^w$ Computer Sciences Department \\ University of Wisconsin \\ Madison, WI 53706 USA}
\date{December 3, 2010; revised February 24, 2012}

\begin{document}

\maketitle

\begin{abstract}
In applications throughout science and engineering one is often faced with the challenge of solving an ill-posed inverse problem, where the number of available measurements is smaller than the dimension of the model to be estimated.  However in many practical situations of interest, models are constrained structurally so that they only have a few degrees of freedom relative to their ambient dimension.  This paper provides a general framework to convert notions of simplicity into convex penalty functions, resulting in convex optimization solutions to linear, underdetermined inverse problems.  The class of simple models considered are those formed as the sum of a few atoms from some (possibly infinite) elementary atomic set; examples include well-studied cases such as sparse vectors (e.g., signal processing, statistics) and low-rank matrices (e.g., control, statistics), as well as several others including sums of a few permutations matrices (e.g., ranked elections, multiobject tracking), low-rank tensors (e.g., computer vision, neuroscience), orthogonal matrices (e.g., machine learning), and atomic measures (e.g., system identification).  The convex programming formulation is based on minimizing the norm induced by the convex hull of the atomic set; this norm is referred to as the \emph{atomic norm}.  The facial structure of the atomic norm ball carries a number of favorable properties that are useful for recovering simple models, and an analysis of the underlying convex geometry provides sharp estimates of the number of generic measurements required for exact and robust recovery of models from partial information.  These estimates are based on computing the Gaussian widths of tangent cones to the atomic norm ball.  When the atomic set has algebraic structure the resulting optimization problems can be solved or approximated via semidefinite programming.  The quality of these approximations affects the number of measurements required for recovery, and this tradeoff is characterized via some examples.  Thus this work extends the catalog of simple models (beyond sparse vectors and low-rank matrices) that can be recovered from limited linear information via tractable convex programming.


\end{abstract}

\section{Introduction}
\label{sec:intro}

Deducing the state or structure of a system from partial, noisy measurements is a fundamental task throughout the sciences and
engineering.  A commonly encountered difficulty that arises in such inverse problems is the limited availability of data relative to the ambient dimension of the signal to be estimated.  However many interesting signals or models in practice contain few degrees of freedom relative to their ambient dimension.  For instance a small number of genes may constitute a signature for disease, very few parameters may be required to specify the correlation structure in a time series, or a sparse collection of geometric constraints might completely specify a molecular configuration.  Such low-dimensional structure plays an important role in making inverse problems well-posed.  In this paper we propose a unified approach to transform notions of simplicity into convex penalty functions, thus obtaining convex optimization formulations for inverse problems.


We describe a model as simple if it can be written as a nonnegative combination of a few elements from an atomic set.  Concretely let $\bx \in \R^p$ be formed as follows:
\begin{equation}
\bx = \sum_{i=1}^k c_i \ba_i, ~~~ \ba_i \in \A, c_i \geq 0, \label{eq:simp1}
\end{equation}
where $\A$ is a set of atoms that constitute simple building blocks of general signals.  Here we assume that $\bx$ is \emph{simple} so that $k$ is relatively small.  For example $\A$ could be the finite set of unit-norm one-sparse vectors in which case $\bx$ is a sparse vector, or $\A$ could be the infinite set of unit-norm rank-one matrices in which case $\bx$ is a low-rank matrix.  These two cases arise in many applications, and have received a tremendous amount of attention recently as several authors have shown that sparse vectors and low-rank matrices can be recovered from highly incomplete information \cite{CanRT2006,Don2006a,Don2006b,RecFP2010,CanR2009}.  However a number of other structured mathematical objects also fit the notion of simplicity described in \eqref{eq:simp1}.  The set $\A$ could be the collection of unit-norm rank-one tensors, in which case $\bx$ is a low-rank tensor and we are faced with the familiar challenge of low-rank tensor decomposition.  Such problems arise in numerous applications in computer vision and image processing \cite{AjaGTL2009}, and in neuroscience \cite{BecS2005}.  Alternatively $\A$ could be the set of permutation matrices; sums of a few permutation matrices are objects of interest in ranking \cite{JagS2010} and multi-object tracking.  As yet another example, $\A$ could consist of measures supported at a single point so that $\bx$ is an atomic measure supported at just a few points.  This notion of simplicity arises in problems in system identification and statistics.

In each of these examples as well as several others, a fundamental problem of interest is to recover $\bx$ given limited \emph{linear} measurements.  For instance the question of recovering a sparse function over the group of permutations (i.e., the sum of a few permutation matrices) given linear measurements in the form of partial Fourier information was investigated in the context of ranked election problems \cite{JagS2010}.  Similar linear inverse problems arise with atomic measures in system identification, with orthogonal matrices in machine learning, and with simple models formed from several other atomic sets (see Section~\ref{subsec:ex} for more examples).  Hence we seek tractable computational tools to solve such problems.  When $\A$ is the collection of one-sparse vectors, a method of choice is to use the $\ell_1$ norm to induce sparse solutions.  This method has seen a surge in interest in the last few years as it provides a tractable convex optimization formulation to exactly recover sparse vectors under various conditions \cite{CanRT2006,Don2006a,Don2006b}.  More recently the nuclear norm has been proposed as an effective convex surrogate for solving rank minimization problems subject to various affine constraints \cite{RecFP2010,CanR2009}.

Motivated by the success of these methods we propose a general convex optimization framework in Section~\ref{sec:def} in order to recover objects with structure of the form \eqref{eq:simp1} from limited linear measurements.  The guiding question behind our framework is:  how do we take a concept of simplicity such as sparsity and derive the $\ell_1$ norm as a convex heuristic?  In other words what is the natural procedure to go from the set of one-sparse vectors $\A$ to the $\ell_1$ norm?  We observe that the convex hull of (unit-Euclidean-norm) one-sparse vectors is the unit ball of the $\ell_1$ norm, or the cross-polytope.  Similarly the convex hull of the (unit-Euclidean-norm) rank-one matrices is the nuclear norm ball; see Figure~\ref{fig:fig1} for illustrations.  These constructions suggest a natural generalization to other settings.  Under suitable conditions the convex hull $\ch(\A)$ defines the unit ball of a norm, which is called the \emph{atomic norm} induced by the atomic set $\A$.  We can then minimize the atomic norm subject to measurement constraints, which results in a convex programming heuristic for recovering simple models given linear measurements.  As an example suppose we wish to recover the sum of a few permutation matrices given linear measurements.  The convex hull of the set of permutation matrices is the \emph{Birkhoff polytope} of doubly stochastic matrices \cite{Zie1995}, and our proposal is to solve a convex program that minimizes the norm induced by this polytope.  Similarly if we wish to recover an orthogonal matrix from linear measurements we would solve a \emph{spectral norm} minimization problem, as the spectral norm ball is the convex hull of all orthogonal matrices.  As discussed in Section~\ref{subsec:why} the atomic norm minimization problem is, in some sense, the best convex heuristic for recovering simple models with respect to a given atomic set.

\begin{figure}
\begin{center}
\subfigure[]{\epsfig{file=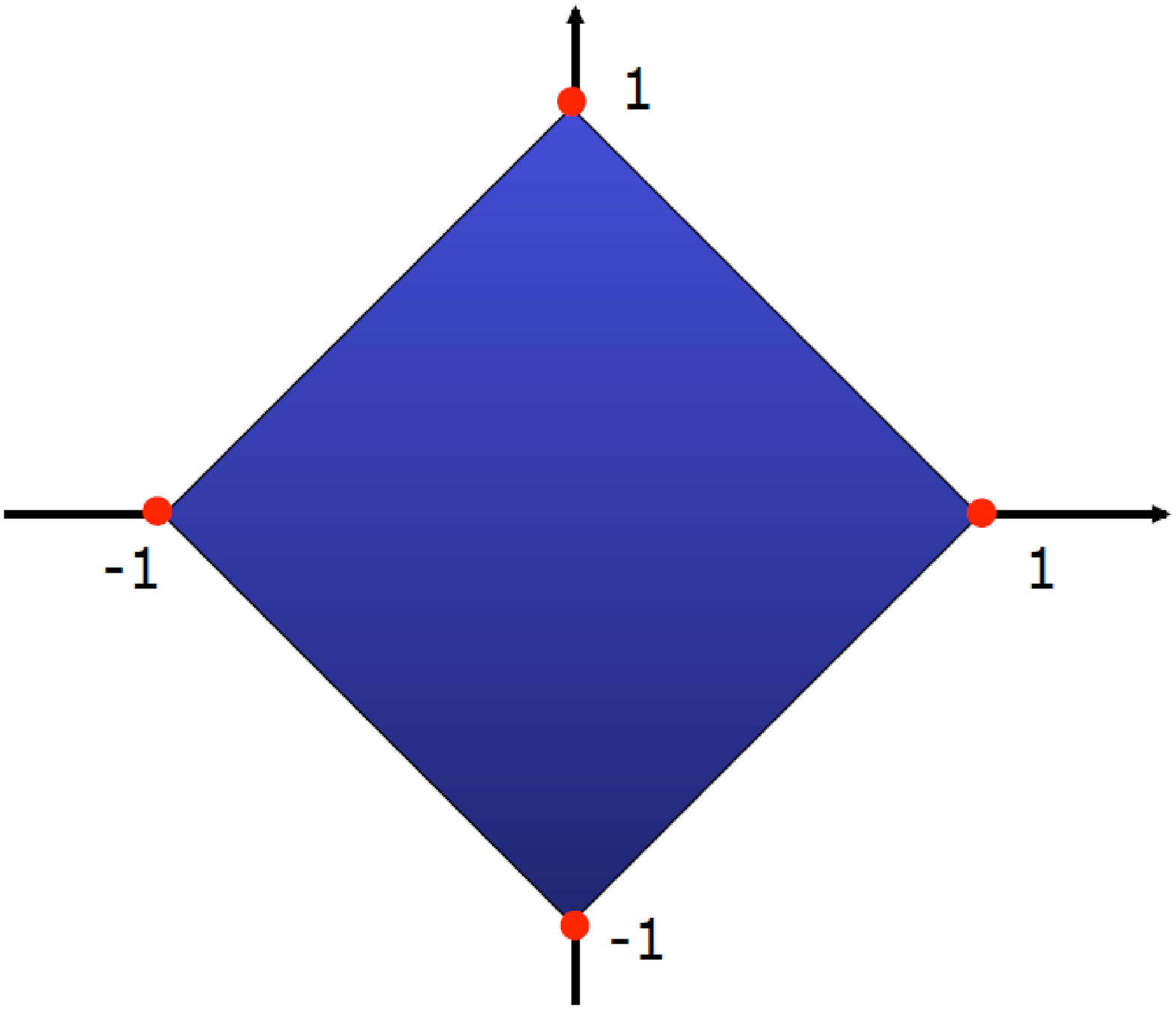,width=4.5cm,height=4cm}}
\hspace{0.3in}
\subfigure[]{\epsfig{file=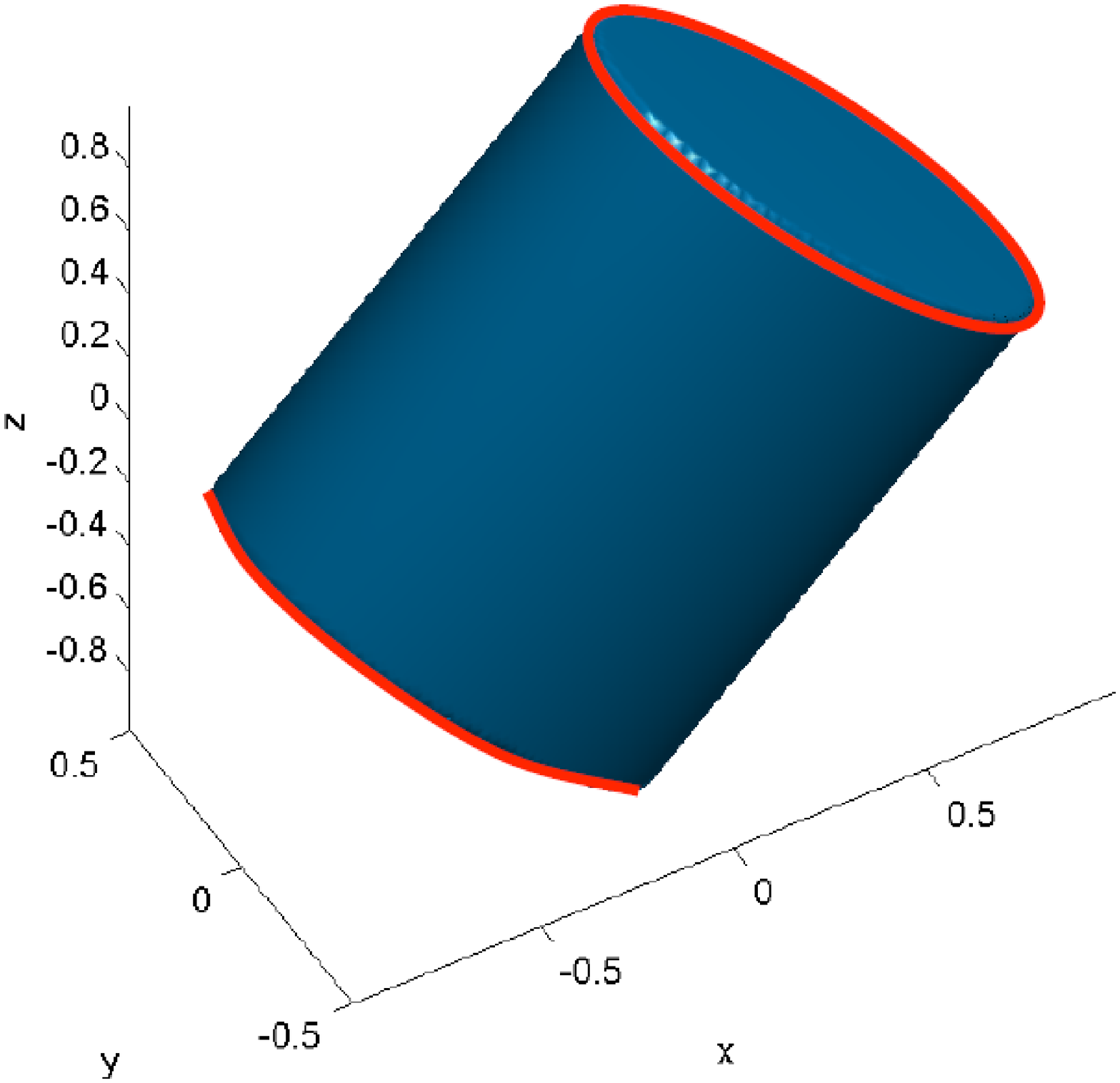,width=4.5cm,height=4cm}}
\hspace{0.3in}
\subfigure[]{\epsfig{file=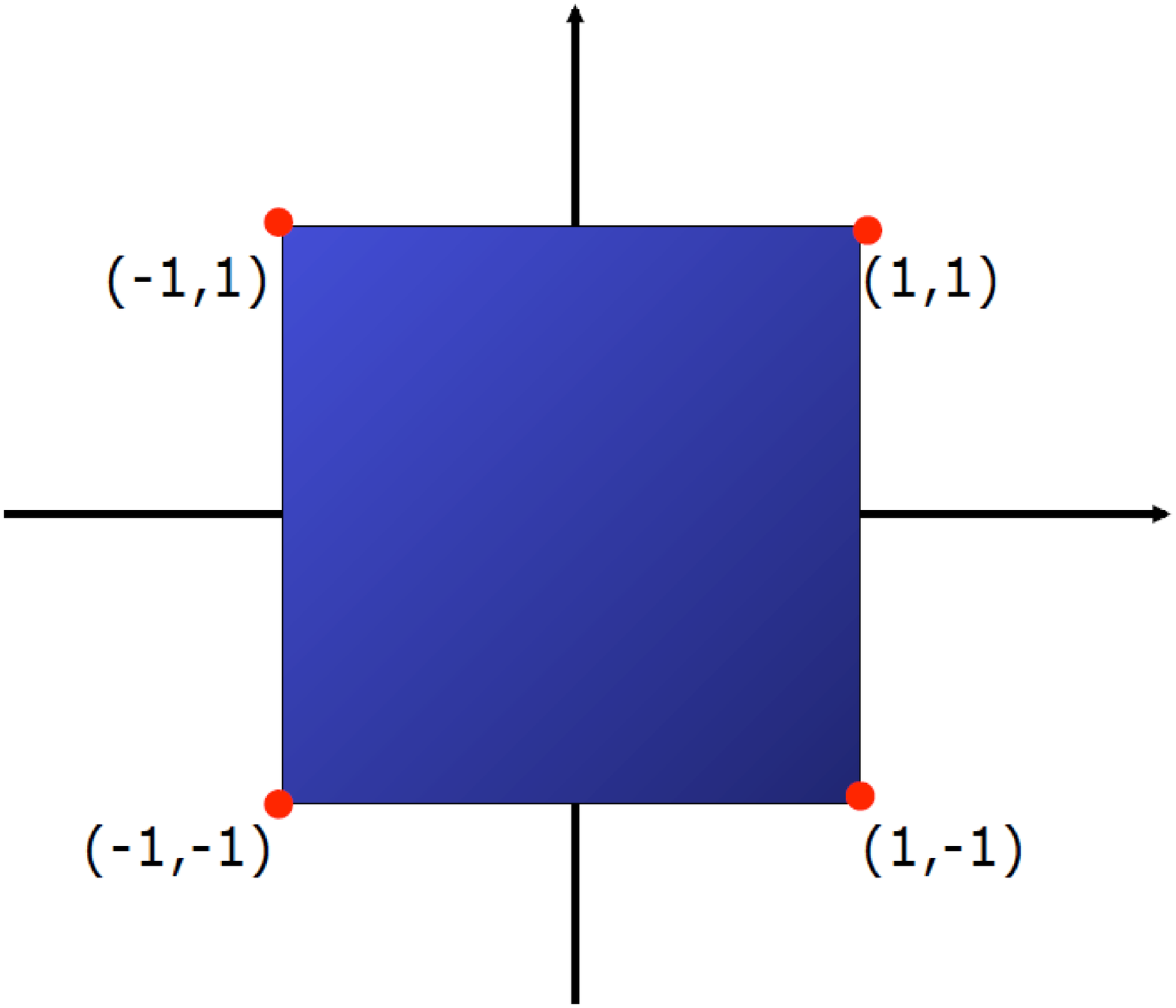,width=4.5cm,height=4cm}}
\caption{Unit balls of some atomic norms: In each figure, the set of atoms is graphed in red and the unit ball of the associated atomic norm is graphed in blue.  In (a), the atoms are the unit-Euclidean-norm one-sparse vectors, and the atomic norm is the $\ell_1$ norm.  In (b), the atoms are the $2 \times 2$ symmetric unit-Euclidean-norm rank-one matrices, and the atomic norm is the nuclear norm.  In (c), the atoms are the vectors $\{-1,+1\}^2$, and the atomic norm is the $\ell_\infty$ norm.} \label{fig:fig1}
\end{center}
\end{figure}

We give general conditions for exact and robust recovery using the atomic norm heuristic.  In Section~\ref{sec:gaussian} we provide concrete bounds on the number of generic linear measurements required for the atomic norm heuristic to succeed.  This analysis is based on computing certain \emph{Gaussian widths} of tangent cones with respect to the unit balls of the atomic norm \cite{Gor1988}.  Arguments based on Gaussian width have been fruitfully applied to obtain bounds on the number of Gaussian measurements for the special case of recovering sparse vectors via $\ell_1$ norm minimization \cite{RudV2006,Sto2009}, but computing Gaussian widths of general cones is not easy.  Therefore it is important to exploit the special structure in atomic norms, while still obtaining sufficiently general results that are broadly applicable.  An important theme in this paper is the connection between Gaussian widths and various notions of \emph{symmetry}.  Specifically by exploiting symmetry structure in certain atomic norms as well as convex duality properties, we give bounds on the number of measurements required for recovery using very general atomic norm heuristics.  For example we provide precise estimates of the number of generic measurements required for exact recovery of an orthogonal matrix via spectral norm minimization, and the number of generic measurements required for exact recovery of a permutation matrix by minimizing the norm induced by the Birkhoff polytope.  While these results correspond to the recovery of individual atoms from random measurements, our techniques are more generally applicable to the recovery of models formed as sums of a few atoms as well.  We also give tighter bounds than those previously obtained on the number of measurements required to robustly recover sparse vectors and low-rank matrices via $\ell_1$ norm and nuclear norm minimization.  In all of the cases we investigate, we find that the number of measurements required to reconstruct an object is proportional to its intrinsic dimension rather than the ambient dimension, thus confirming prior folklore.  See Table~\ref{tab:sum} for a summary of these results.

\begin{table}[t]
\centering
\begin{tabular}{||c|c|c||}\hline
 Underlying model & Convex heuristic & $\#$ Gaussian measurements \\ \hline\hline
 $s$-sparse vector in $\R^p$ & $\ell_1$ norm & $2s \log(p/s)+ 5s/4 $ \\ \hline
 $m \times m$ rank-$r$ matrix & nuclear norm & $3r(2m-r)$ \\
 \hline
 sign-vector $\{-1,+1\}^p$ & $\ell_\infty$ norm & $p / 2$\\ \hline
 $m \times m$ permutation matrix & norm induced by Birkhoff polytope & $9 \,m \log (m)$\\
 \hline
 $m \times m$ orthogonal matrix & spectral norm & $(3 m^2 - m)/4$\\
 \hline
\end{tabular}
\caption{A summary of the recovery bounds obtained using Gaussian width arguments.} \label{tab:sum}
\end{table}

Although our conditions for recovery and bounds on the number of measurements hold generally, we note that it may not be possible to obtain a computable representation for the convex hull $\ch(\A)$ of an arbitrary set of atoms $\A$.  This leads us to another important theme of this paper, which we discuss in Section~\ref{sec:rep}, on the connection between algebraic structure in $\A$ and the semidefinite representability of the convex hull $\ch(\A)$.  In particular when $\A$ is an algebraic variety the convex hull $\ch(\A)$ can be approximated as (the projection of) a set defined by linear matrix inequalities.  Thus the resulting atomic norm minimization heuristic can be solved via semidefinite programming.  A second issue that arises in practice is that even with algebraic structure in $\A$ the semidefinite representation of $\ch(\A)$ may not be computable in polynomial time, which makes the atomic norm minimization problem intractable to solve.  A prominent example here is the tensor nuclear norm ball, obtained by taking the convex hull of the rank-one tensors.  In order to address this problem we give a hierarchy of semidefinite relaxations using \emph{theta bodies} that approximate the original (intractable) atomic norm minimization problem \cite{GouPT2010}.  We also highlight that while these semidefinite relaxations are more tractable to solve, we require more measurements for exact recovery of the underlying model than if we solve the original intractable atomic norm minimization problem.  Hence there is a tradeoff between the complexity of the recovery algorithm and the number of measurements required for recovery.  We illustrate this tradeoff with the cut polytope and its relaxations.

\textbf{Outline} Section~\ref{sec:def} describes the construction of the atomic norm, gives several examples of applications in which these norms may be useful to recover simple models, and provides general conditions for recovery by minimizing the atomic norm.  In Section~\ref{sec:gaussian} we investigate the number of generic measurements for exact or robust recovery using atomic norm minimization, and give estimates in a number of settings by analyzing the Gaussian width of certain tangent cones.  We address the problem of semidefinite representability and tractable relaxations of the atomic norm in Section~\ref{sec:rep}.  Section~\ref{sec:comp} describes some algorithmic issues as well as a few simulation results, and we conclude with a discussion and open questions in Section~\ref{sec:conc}.



\section{Atomic Norms and Convex Geometry}
\label{sec:def}

In this section we describe the construction of an atomic norm from a collection of simple atoms.  In addition we give several examples of atomic norms, and discuss their properties in the context of solving ill-posed linear inverse problems.  We denote the Euclidean norm by $\|\cdot\|$.

\subsection{Definition}
\label{subsec:def}

Let $\A$ be a collection of atoms that is a compact subset of $\R^p$.  We will assume throughout this paper that no element $\ba \in \A$ lies in the convex hull of the other elements $\ch(\A \backslash \ba)$, i.e., the elements of $\A$ are the extreme points of $\ch(\A)$.  Let $\|\bx\|_\A$ denote the gauge of $\A$ \cite{Roc1996}:
\begin{equation}
	\|\bx\|_{\A} = \inf\{t>0 ~:~ x \in t~ \ch(\A)\}. \label{eq:atnorm}
\end{equation}
Note that the gauge is always a convex, extended-real valued function for any set $\A$.  By convention this function evaluates to $+\infty$ if $\bx$ does not lie in the affine hull of $\ch(\A)$.  We will assume without loss of generality that the centroid of $\ch(\A)$ is at the origin, as this can be achieved by appropriate recentering.  With this assumption the gauge function can be rewritten as \cite{Bon1991}:
\begin{equation*}
\|\bx\|_{\A} = \inf\left\{ \sum_{\ba \in\A} c_\ba  ~:~ \bx = \sum_{\ba\in \A} c_\ba \ba, ~~ c_\ba \geq 0 ~\forall \ba \in \A\right\}.
\end{equation*}
If $\A$ is centrally symmetric about the origin (i.e., $\ba \in \A$ if and only if $-\ba \in \A$) we have that $\|\cdot\|_\A$ is a norm, which we call the \emph{atomic norm} induced by $\A$.  The support function of $\A$ is given as:
\begin{equation}
\|\bx\|_\A^\ast = \sup \left\{\langle \bx, \ba \rangle ~ : ~ \ba \in \A  \right\}. \label{eq:atnormd}
\end{equation}
If $\|\cdot\|_\A$ is a norm the support function $\|\cdot\|^\ast_\A$ is the dual norm of this atomic norm.  From this definition we see that the unit ball of $\|\cdot\|_\A$ is equal to $\ch(\A)$.  In many examples of interest the set $\A$ is not centrally symmetric, so that the gauge function does not define a norm.  However our analysis is based on the underlying convex geometry of $\ch(\A)$, and our results are applicable even if $\|\cdot\|_\A$ does not define a norm. Therefore, with an abuse of terminology we generally refer to $\|\cdot\|_\A$ as the atomic norm of the set $\A$ even if $\|\cdot\|_\A$ is not a norm.  We note that the duality characterization between \eqref{eq:atnorm} and \eqref{eq:atnormd} when $\|\cdot\|_\A$ is a norm is in fact applicable even in infinite-dimensional Banach spaces by Bonsall's atomic decomposition theorem \cite{Bon1991}, but our focus is on the finite-dimensional case in this work.  We investigate in greater detail the issues of representability and efficient approximation of these atomic norms in Section~\ref{sec:rep}.


Equipped with a convex penalty function given a set of atoms, we propose a convex optimization method to recover a ``simple'' model given limited linear measurements.  Specifically suppose that $\bxs$ is formed according to \eqref{eq:simp1} from a set of atoms $\A$.  Further suppose that we have a known linear map $\Phi: \R^p \rightarrow \R^n$, and we have linear information about $\bxs$ as follows:
\begin{equation}
\by = \Phi \bxs.
\end{equation}
The goal is to reconstruct $\bxs$ given $\by$.  We consider the following convex formulation to accomplish this task:
\begin{equation}
\begin{aligned}
\hat{\bx} = \arg \min_{\bx} & ~~~ \|\bx\|_\A \\ \mbox{s.t.} & ~~~ \by = \Phi \bx.
\end{aligned}
\label{eq:atomic-norm-primal}
\end{equation}
When $\A$ is the set of one-sparse atoms this problem reduces to standard $\ell_1$ norm minimization.  Similarly when $\A$ is the set of rank-one matrices this problem reduces to nuclear norm minimization.  More generally if the atomic norm $\|\cdot\|_\A$ is tractable to evaluate, then \eqref{eq:atomic-norm-primal} potentially offers an efficient convex programming formulation for reconstructing $\bxs$ from the limited information $\by$.  The \emph{dual problem} of \eqref{eq:atomic-norm-primal} is given as follows:
\begin{equation}
\begin{aligned}
\max_{\bz} & ~~~ \by^T \bz \\ \mbox{s.t.} & ~~~ \|\Phi^\dag \bz\|^\ast_\A \leq 1.
\end{aligned}
\label{eq:atomic-norm-dual}
\end{equation}
Here $\Phi^\dag$ denotes the adjoint (or transpose) of the linear measurement map $\Phi$.

The convex formulation \eqref{eq:atomic-norm-primal} can be suitably modified in case we only have access to inaccurate, noisy information.  Specifically suppose that we have noisy measurements $\by = \Phi \bxs + \omega$ where $\omega$ represents the noise term.  A natural convex formulation is one in which the constraint $\by = \Phi \bx$ of \eqref{eq:atomic-norm-primal} is replaced by the relaxed constraint $\|\by - \Phi \bx\| \leq \delta$, where $\delta$ is an upper bound on the size of the noise $\omega$:
\begin{equation}
\begin{aligned}
\hat{\bx} = \arg \min_{\bx} & ~~~ \|\bx\|_\A \\ \mbox{s.t.} & ~~~ \|\by - \Phi \bx\| \leq \delta.
\end{aligned}
\label{eq:noisy-atomic-norm-primal}
\end{equation}

We say that we have \emph{exact recovery} in the noise-free case if $\hat{\bx} = \bxs$ in \eqref{eq:atomic-norm-primal}, and \emph{robust recovery} in the noisy case if the error $\|\hat{\bx}-\bxs\|$ is small in \eqref{eq:noisy-atomic-norm-primal}.  In Section~\ref{subsec:reccond} and Section~\ref{sec:gaussian} we give conditions under which the atomic norm heuristics \eqref{eq:atomic-norm-primal} and \eqref{eq:noisy-atomic-norm-primal} recover $\bxs$ exactly or approximately.  Atomic norms have found fruitful applications in problems in approximation theory of various function classes \cite{Pis1981,Jon1992,Bar1993,DeVT1996}.  However this prior body of work was concerned with infinite-dimensional Banach spaces, and none of these references consider nor provide recovery guarantees that are applicable in our setting.

%

\subsection{Examples}
\label{subsec:ex}

Next we provide several examples of atomic norms that can be viewed as special cases of the construction above.  These norms are obtained by convexifying atomic sets that are of interest in various applications.

\textbf{Sparse vectors.}  The problem of recovering sparse vectors from limited measurements has received a great deal of attention, with applications in many problem domains.  In this case the atomic set $\A \subset \mathbb{R}^p$ can be viewed as the set of unit-norm one-sparse vectors $\{\pm \mathbf{e}_i\}_{i=1}^p$, and $k$-sparse vectors in $\mathbb{R}^p$ can be constructed using a linear combination of $k$ elements of the atomic set.  In this case it is easily seen that the convex hull $\ch(\A)$ is given by the \emph{cross-polytope} (i.e., the unit ball of the $\ell_1$ norm), and the atomic norm $\|\cdot\|_\A$ corresponds to the $\ell_1$ norm in $\mathbb{R}^p$.

\textbf{Low-rank matrices.}  Recovering low-rank matrices from limited information is also a problem that has received considerable attention as it finds applications in problems in statistics, control, and machine learning.  The atomic set $\A$ here can be viewed as the set of rank-one matrices of unit-Euclidean-norm.  The convex hull $\ch(\A)$ is the \emph{nuclear norm ball} of matrices in which the sum of the singular values is less than or equal to one.

\textbf{Sparse and low-rank matrices.} The problem of recovering a sparse matrix and a low-rank matrix given information about their sum arises in a number of model selection and system identification settings.  The corresponding atomic norm is constructed by taking the convex hull of an atomic set obtained via the union of rank-one matrices and (suitably scaled) one-sparse matrices.  This norm can also be viewed as the \emph{infimal convolution} of the $\ell_1$ norm and the nuclear norm, and its properties have been explored in \cite{ChaSPW2011,CanLMW2011}.

\textbf{Permutation matrices.}  A problem of interest in a ranking context \cite{JagS2010} or an object tracking context is that of recovering permutation matrices from partial information.  Suppose that a small number $k$ of rankings of $m$ candidates is preferred by a population.  Such preferences can be modeled as the sum of a few $m \times m$ permutation matrices, with each permutation corresponding to a particular ranking.  By conducting surveys of the population one can obtain partial linear information of these preferred rankings.  The set $\A$ here is the collection of permutation matrices (consisting of $m!$ elements), and the convex hull $\ch(\A)$ is the \emph{Birkhoff polytope} or the set of doubly stochastic matrices \cite{Zie1995}.  The centroid of the Birkhoff polytope is the matrix $\ones \ones^T / m$, so it needs to be recentered appropriately.  We mention here recent work by Jagabathula and Shah \cite{JagS2010} on recovering a sparse function over the symmetric group (i.e., the sum of a few permutation matrices) given partial Fourier information; although the algorithm proposed in \cite{JagS2010} is tractable it is not based on convex optimization.

\textbf{Binary vectors.}  In integer programming one is often interested in recovering vectors in which the entries take on values of $\pm 1$.  Suppose that there exists such a sign-vector, and we wish to recover this vector given linear measurements.  This corresponds to a version of the multi-knapsack problem \cite{ManR2009}.  In this case $\A$ is the set of all sign-vectors, and the convex hull $\ch(\A)$ is the \emph{hypercube} or the unit ball of the $\ell_\infty$ norm.  The image of this hypercube under a linear map is also referred to as a zonotope \cite{Zie1995}.

\textbf{Vectors from lists.}  Suppose there is an unknown vector $\bx \in \mathbb{R}^p$, and that we are given the entries of this vector without any information about the locations of these entries.  For example if $\bx = [3 ~ 1 ~ 2 ~ 2 ~ 4]'$, then we are only given the list of numbers $\{1, 2, 2, 3, 4\}$ without their positions in $\bx$.  Further suppose that we have access to a few linear measurements of $\bx$.  Can we recover $\bx$ by solving a convex program?  Such a problem is of interest in recovering partial rankings of elements of a set.  An extreme case is one in which we only have two preferences for rankings, i.e., a vector in $\{1,2\}^p$ composed only of one's and two's, which reduces to a special case of the problem above of recovering binary vectors (in which the number of entries of each sign is fixed). For this problem the set $\A$ is the set of all permutations of $\bx$ (which we know since we have the list of numbers that compose $\bx$), and the convex hull $\ch(\A)$ is the \emph{permutahedron} \cite{Zie1995,SanSS2009}.  As with the Birkhoff polytope, the permutahedron also needs to be recentered about the point $\ones^T \bx / p$.

\textbf{Matrices constrained by eigenvalues.}  This problem is in a sense the non-commutative analog of the one above.  Suppose that we are given the eigenvalues $\lambda$ of a symmetric matrix, but no information about the eigenvectors.  Can we recover such a matrix given some additional linear measurements?  In this case the set $\A$ is the set of all symmetric matrices with eigenvalues $\lambda$, and the convex hull $\ch(\A)$ is given by the \emph{Schur-Horn orbitope} \cite{SanSS2009}.

\textbf{Orthogonal matrices.}  In many applications matrix variables are constrained to be orthogonal, which is a non-convex constraint and may lead to computational difficulties.  We consider one such simple setting in which we wish to recover an orthogonal matrix given limited information in the form of linear measurements.  In this example the set $\A$ is the set of $m \times m$ orthogonal matrices, and $\ch(\A)$ is the \emph{spectral norm ball}.

\textbf{Measures.} Recovering a measure given its moments is another question of interest that arises in system identification and statistics.  Suppose one is given access to a linear combination of moments of an atomically supported measure.  How can we reconstruct the support of the measure?  The set $\A$ here is the moment curve, and its convex hull $\ch(\A)$ goes by several names including the \emph{Caratheodory orbitope} \cite{SanSS2009}.  Discretized versions of this problem correspond to the set $\A$ being a finite number of points on the moment curve;  the convex hull $\ch(\A)$ is then a \emph{cyclic polytope} \cite{Zie1995}.

\textbf{Cut matrices.}  In some problems one may wish to recover low-rank matrices in which the entries are constrained to take on values of $\pm 1$.  Such matrices can be used to model basic user preferences, and are of interest in problems such as collaborative filtering \cite{SreS2005}.  The set of atoms $\A$ could be the set of rank-one signed matrices, i.e., matrices of the form $\bz \bz^T$ with the entries of $\bz$ being $\pm 1$.  The convex hull $\ch(\A)$ of such matrices is the \emph{cut polytope} \cite{DezL1997}.  An interesting issue that arises here is that the cut polytope is in general intractable to characterize.  However there exist several well-known tractable semidefinite relaxations to this polytope \cite{DezL1997,GoeW1995}, and one can employ these in constructing efficient convex programs for recovering cut matrices.  We discuss this point in greater detail in Section~\ref{subsec:tradeoff}.

\textbf{Low-rank tensors.} Low-rank tensor decompositions play an important role in numerous applications throughout signal processing and machine learning \cite{KolB2009}.  Developing computational tools to recover low-rank tensors is therefore of great interest.  In principle we could solve a tensor nuclear norm minimization problem, in which the tensor nuclear norm ball is obtained by taking the convex hull of rank-one tensors.  A computational challenge here is that the tensor nuclear norm is in general intractable to compute;  in order to address this problem we discuss further convex relaxations to the tensor nuclear norm using theta bodies in Section~\ref{sec:rep}.  A number of additional technical issues also arise with low-rank tensors including the non-existence in general of a singular value decomposition analogous to that for matrices \cite{Kol2001}, and the difference between the rank of a tensor and its border rank \cite{deSL2008}.

\textbf{Nonorthogonal factor analysis.}  Suppose that a data matrix admits a factorization $X=AB$.  The matrix nuclear norm heuristic will find a factorization into \emph{orthogonal} factors in which the columns of $A$ and rows of $B$ are mutually orthogonal.  However if \emph{a priori} information is available about the factors, precision and recall could be improved by enforcing such priors.  These priors may sacrifice orthogonality, but the factors might better conform with assumptions about how the data are generated. For instance in some applications one might know in advance that the factors should only take on a discrete set of values~\cite{SreS2005}.  In this case, we might try to fit a sum of rank-one matrices that are bounded in $\ell_\infty$ norm rather than in $\ell_2$ norm.  Another prior that commonly arises in practice is that the factors are non-negative (i.e., in non-negative matrix factorization).  These and other priors on the basic rank-one summands induce different norms on low-rank models than the standard nuclear norm \cite{Faz2002}, and may be better suited to specific applications.


\subsection{Background on Tangent and Normal Cones}
\label{subsec:bgtn}

In order to properly state our results, we recall some basic concepts from convex analysis. A convex set $\mathcal{C}$ is a \emph{cone} if it is closed under positive linear combinations.  The polar $\mathcal{C}^\ast$ of a cone $\mathcal{C}$ is the cone
\begin{equation*}
\mathcal{C}^\ast = \{ x \in \R^p ~:~ \langle x,z\rangle \leq 0~\forall z\in \mathcal{C}\}.
\end{equation*}
Given some nonzero $\bx \in \R^p$ we define the \emph{tangent cone} at $\bx$ with respect to the scaled unit ball $\|\bx\|_\A \ch(\A)$ as
\begin{equation}
T_\A(\bx) = \mathrm{cone}\{\mathbf{z}-\bx~:~ \|\mathbf{z}\|_\A \leq \|\bx\|_\A\}. \label{eq:tcone}
\end{equation}
The cone $T_\A(\bx)$ is equal to the set of \emph{descent directions} of the atomic norm $\|\cdot\|_\A$ at the point $\bx$, i.e., the set of all directions $\mathbf{d}$ such that the directional derivative is negative.


The \emph{normal cone} $N_{\A}(\bx)$ at $\bx$ with respect to the scaled unit ball $\|\bx\|_\A \ch(\A)$ is defined to be the set of all directions $\mathbf{s}$ that form obtuse angles with every descent direction of the atomic norm $\|\cdot\|_\A$ at the point $\bx$:
\begin{equation}
N_\A(\bx) = \{\mathbf{s} : \langle \mathbf{s},\mathbf{z}-\bx \rangle \leq 0 ~ \forall \mathbf{z} ~ \mathrm{s.t.~} \|\mathbf{z}\|_\A\leq \|\bx\|_\A\}. \label{eq:ncone}
\end{equation}
The normal cone is equal to the set of all normals of hyperplanes given by normal vectors $\mathbf{s}$ that support the scaled unit ball $\|\bx\|_\A \ch(\A)$ at $\bx$.  Observe that the polar cone of the tangent cone $T_\A(\bx)$ is the normal cone $N_\A(\bx)$ and vice-versa.  Moreover we have the basic characterization that the normal cone $N_{\A}(\bx)$ is the conic hull of the subdifferential of the atomic norm at $\bx$.

\subsection{Recovery Condition}
\label{subsec:reccond}

The following result gives a characterization of the favorable underlying geometry required for exact recovery.  Let $\n(\Phi)$ denote the nullspace of the operator $\Phi$.

\begin{proposition}\label{prop:null-intersection}
We have that $\hat{\bx} = \bxs$ is the unique optimal solution of \eqref{eq:atomic-norm-primal} if and only if $\n(\Phi) \cap T_{\A}(\bxs) = \{0\}$.
\end{proposition}

\begin{proof}
Eliminating the equality constraints in \eqref{eq:atomic-norm-primal} we have the equivalent optimization problem
\begin{equation*}
\min_{\mathbf{d}} ~\|\bxs + \mathbf{d}\|_\A \quad \mathrm{s.t.}~ \mathbf{d} \in \n(\Phi).
\end{equation*}
Suppose  $\n(\Phi)\cap T_{\A}(\bxs)=\{0\}$.  Since $\|\bxs + \mathbf{d}\|_\A\leq \|\bxs\|$ implies $\mathbf{d} \in T_{\A}(\bxs)$, we have that $\|\bxs + \mathbf{d}\|_\A > \|\bxs\|_\A$ for all $\mathbf{d} \in \n(\Phi)\setminus\{0\}$.   Conversely $\bxs$ is the unique optimal solution of \eqref{eq:atomic-norm-primal} if $\|\bxs + \mathbf{d}\|_\A > \|\bxs\|_\A$ for all $\mathbf{d} \in \n(\Phi)\setminus\{0\}$, which implies that $\mathbf{d} \not\in T_{\A}(\bxs)$.
\end{proof}

Proposition~\ref{prop:null-intersection} asserts that  the atomic norm heuristic succeeds if the nullspace of the sampling operator does not intersect the tangent cone $T_\A(\bxs)$ at $\bxs$.  In Section~\ref{sec:gaussian} we provide a characterization of tangent cones that determines the number of Gaussian measurements required to guarantee such an empty intersection.

A tightening of this empty intersection condition can also be used to address the noisy approximation problem.  The following proposition characterizes when $\bxs$ can be \emph{well-approximated} using the convex program \eqref{eq:noisy-atomic-norm-primal}.

\begin{proposition}\label{prop:noisy-recovery}
Suppose that we are given $n$ noisy measurements $\by = \Phi \bxs + \omega$ where $\| \omega\| \leq \delta$, and $\Phi: \R^p \rightarrow \R^n$.  Let $\hat{\bx}$ denote an optimal solution of \eqref{eq:noisy-atomic-norm-primal}.  Further suppose for all $\bz \in T_{\A}(\bxs)$ that we have $\|\Phi \bz\| \geq \epsilon\|\bz\|$.  Then $\|\hat{\bx}-\bxs\| \leq \frac{2\delta}{\epsilon}$.
\end{proposition}

\begin{proof}
The set of descent directions at $\bxs$ with respect to the atomic norm ball is given by the tangent cone $T_{\A}(\bxs)$.  The error vector $\hat{\bx} - \bxs$ lies in $T_{\A}(\bxs)$ because $\hat{\bx}$ is a minimal atomic norm solution, and hence $\|\hat{\bx}\|_\A \leq \|\bxs\|_\A$.  It follows by the triangle inequality that
\begin{equation}
\|\Phi(\hat{\bx}-\bxs)\| \leq \|\Phi\hat{\bx} - \by\| + \|\Phi \bxs - \by\| \leq 2 \delta.
\end{equation}
By assumption we have that
\begin{equation}
\|\Phi(\hat{\bx}-\bxs)\| \geq \epsilon \|\hat{\bx}-\bxs\|,
\end{equation}
which allows us to conclude that $\|\hat{\bx}-\bxs\|\leq \frac{2\delta}{\epsilon}$.
\end{proof}
Therefore, we need only concern ourselves with estimating the minimum value of $\frac{\|\Phi \bz\|}{\|\bz\|}$ for non-zero $\bz\in T_{\A}(\bxs)$.  We denote this quantity as the \emph{minimum gain} of the measurement operator $\Phi$ restricted to the cone $T_{\A}(\bxs)$.  In particular if this minimum gain is bounded away from zero, then the atomic norm heuristic also provides robust recovery when we have access to noisy linear measurements of $\bxs$.  Minimum gain conditions have been employed in recent recovery results via $\ell_1$-norm minimization, block-sparse vector recovery, and low-rank matrix reconstruction \cite{CanT2005,BicRT2009,VanB2009,NegRWY2010}.   All of these results rely heavily on strong decomposability conditions of the $\ell_1$ norm and the matrix nuclear norm.  However, there are several examples of atomic norms (for instance, the $\ell_\infty$ norm and the norm induced by the Birkhoff polytope) as specified in Section~\ref{subsec:ex} that do not satisfy such decomposability conditions. As we'll see in the sequel the more geometric viewpoint adopted in this paper provides a fruitful framework in which to analyze the recovery properties of general atomic norms.
%


\subsection{Why Atomic Norm?}
\label{subsec:why}

The atomic norm induced by a set $\A$ possesses a number of favorable properties that are useful for recovering ``simple'' models from limited linear measurements.  The key point to note from Section~\ref{subsec:reccond} is that the smaller the tangent cone at a point $\bxs$ with respect to $\ch(\A)$, the easier it is to satisfy the empty-intersection condition of Proposition~\ref{prop:null-intersection}.

Based on this observation it is desirable that points in $\ch(\A)$ with smaller tangent cones correspond to simpler models, while points in $\ch(\A)$ with larger tangent cones generally correspond to more complicated models.  The construction of $\ch(\A)$ by taking the convex hull of $\A$ ensures that this is the case.  The extreme points of $\ch(\A)$ correspond to the simplest models, i.e., those models formed from a single element of $\A$.  Further the low-dimensional faces of $\ch(\A)$ consist of those elements that are obtained by taking linear combinations of a few basic atoms from $\A$.  These are precisely the properties desired as points lying in these low-dimensional faces of $\ch(\A)$ have smaller tangent cones than those lying on larger faces.

We also note that the atomic norm is, in some sense, the best possible convex heuristic for recovering simple models.  Any reasonable heuristic penalty function should be constant on the set of atoms $\A$.  This ensures that no atom is preferred over any other.  Under this assumption, we must have that for any $\ba \in \A$, $\ba'-\ba$ must be a descent direction for all $\ba'\in\A$.  The best convex penalty function is one in which the cones of descent directions at $\ba\in\A$ are as small as possible.  This is because, as described above, smaller cones  are more likely to satisfy the empty intersection condition required for exact recovery.  Since the tangent cone at $\ba \in \A$ with respect to $\ch(\A)$ is precisely the conic hull of $\ba'-\ba$ for $\ba'\in \A$, the atomic norm is the best convex heuristic for recovering models where simplicity is dictated by the set $\A$.


Our reasons for proposing the atomic norm as a useful convex heuristic are quite different from previous justifications of the $\ell_1$ norm and the nuclear norm.  In particular let $f: \R^p \rightarrow \R$ denote the cardinality function that counts the number of nonzero entries of a vector.  Then the $\ell_1$ norm is the \emph{convex envelope} of $f$ restricted to the unit ball of the $\ell_\infty$ norm, i.e., the best convex underestimator of $f$ restricted to vectors in the $\ell_\infty$-norm ball.  This view of the $\ell_1$ norm in relation to the function $f$ is often given as a justification for its effectiveness in recovering sparse vectors.  However if we consider the convex envelope of $f$ restricted to the Euclidean norm ball, then we obtain a very different convex function than the $\ell_1$ norm!  With more general atomic sets, it may not be clear \emph{a priori} what the bounding set should be in deriving the convex envelope.  In contrast the viewpoint adopted in this paper leads to a natural, unambiguous construction of the $\ell_1$ norm and other general atomic norms.  Further as explained above it is the favorable \emph{facial structure} of the atomic norm ball that makes the atomic norm a suitable convex heuristic to recover simple models, and this connection is transparent in the definition of the atomic norm.


\section{Recovery from Generic Measurements}
\label{sec:gaussian}

We consider the question of using the convex program \eqref{eq:atomic-norm-primal} to recover ``simple'' models formed according to \eqref{eq:simp1} from a \emph{generic} measurement operator or map $\Phi: \R^p \rightarrow \R^n$.  Specifically, we wish to compute estimates on the number of measurements $n$ so that we have exact recovery using \eqref{eq:atomic-norm-primal} for \emph{most} operators comprising of $n$ measurements.  That is, the measure of $n$-measurement operators for which recovery fails using \eqref{eq:atomic-norm-primal} must be exponentially small.  In order to conduct such an analysis we study random \emph{Gaussian} maps whose entries are independent and identically distributed Gaussians.  These measurement operators have a nullspace that is uniformly distributed among the set of all $(p-n)$-dimensional subspaces in $\R^p$.  In particular we analyze when such operators satisfy the conditions of Proposition~\ref{prop:null-intersection} and Proposition~\ref{prop:noisy-recovery} for exact recovery.

\subsection{Recovery Conditions based on Gaussian Width}
\label{subsec:width}


Proposition~\ref{prop:null-intersection} requires that the nullspace of the measurement operator $\Phi$ must miss the tangent cone $T_{\A}(\bxs)$.  Gordon \cite{Gor1988} gave a solution to the problem of characterizing the probability that a random subspace (of some fixed dimension) distributed uniformly misses a cone.  We begin by defining the Gaussian width of a set, which plays a key role in Gordon's analysis.
\begin{definition}
The \emph{Gaussian width} of a set $S \subset \R^p$ is defined as:
\begin{equation*}
w(S) := \E_\bg\left[\sup_{\bz \in S} ~ \bg^T \bz \right],
\end{equation*}
where $\bg \sim \mathcal{N}(0,I)$ is a vector of independent zero-mean unit-variance Gaussians.
\end{definition}
Gordon characterized the likelihood that a random subspace misses a cone $\mathcal{C}$ purely in terms of the dimension of the subspace and the Gaussian width $w(\mathcal{C} \cap \Sp^{p-1})$, where $\Sp^{p-1} \subset \R^p$ is the unit sphere.  Before describing Gordon's result formally, we introduce some notation.  Let $\lambda_k$ denote the expected length of a $k$-dimensional Gaussian random vector.  By elementary integration, we have that  $\lambda_k = \sqrt{2}\Gamma(\tfrac{k+1}{2})/\Gamma(\tfrac{k}{2})$. Further by induction one can show that $\lambda_k$ is tightly bounded as $\frac{k}{\sqrt{k+1}}\leq \lambda_k \leq \sqrt{k}$.

The main idea underlying Gordon's theorem is a bound on the minimum gain of an operator restricted to a set.  Specifically, recall that $\n(\Phi) \cap T_\A(\bxs) = \{0\}$ is the condition required for recovery by Proposition~\ref{prop:null-intersection}.  Thus if we have that the minimum gain of $\Phi$ restricted to vectors in the set $T_\A(\bxs) \cap \Sp^{p-1}$ is bounded away from zero, then it is clear that $\n(\Phi) \cap T_\A(\bxs) = \emptyset$.  We refer to such minimum gains restricted to a subset of the sphere as \emph{restricted minimum singular values}, and the following theorem of Gordon gives a bound these quantities:

%

\begin{theorem}[Cor.~1.2,~\cite{Gor1988}] \label{theo:escape}
Let $\Omega$ be a closed subset of $\Sp^{p-1}$.  Let $\Phi: \R^p \rightarrow \R^n$ be a random map with i.i.d. zero-mean Gaussian entries having variance one. Then
\begin{equation}
	\E\left[ \min_{\bz\in \Omega} \|\Phi \bz\|_2\right]  \geq \lambda_n - w(\Omega)\,.
\end{equation}
\end{theorem}

Theorem~\ref{theo:escape} allows us to characterize exact recovery in the noise-free case using the convex program \eqref{eq:atomic-norm-primal}, and robust recovery in the noisy case using the convex program \eqref{eq:noisy-atomic-norm-primal}.  Specifically, we consider the number of measurements required for exact or robust recovery when the measurement map $\Phi: \R^p \rightarrow \R^n$ consists of i.i.d. zero-mean Gaussian entries having variance $1/n$.  The normalization of the variance ensures that the columns of $\Phi$ are approximately unit-norm, and is necessary in order to properly define a signal-to-noise ratio.  The following corollary summarizes the main results of interest in our setting:

\begin{corollary} \label{corl:width}
Let $\Phi: \R^p \rightarrow \R^n$ be a random map with i.i.d. zero-mean Gaussian entries having variance $1/n$.  Further let $\Omega = T_\A(\bxs) \cap \Sp^{p-1}$ denote the spherical part of the tangent cone $T_\A(\bxs)$.

\begin{enumerate}
\item  Suppose that we have measurements $\by = \Phi \bxs$ and solve the convex program \eqref{eq:atomic-norm-primal}. Then $\bxs$ is the unique optimum of \eqref{eq:atomic-norm-primal} with probability at least $1-\exp\left(-\tfrac{1}{2}\left[\lambda_n - w(\Omega)\right]^2\right)$ provided	\begin{equation*}
	n \geq w(\Omega)^2+1\,.
	\end{equation*}

\item Suppose that we have noisy measurements $\by = \Phi \bxs + \omega$, with the noise $\omega$ bounded as $\|\omega\| \leq \delta$, and that we solve the convex program \eqref{eq:noisy-atomic-norm-primal}.  Letting $\hat{\bx}$ denote the optimal solution of \eqref{eq:noisy-atomic-norm-primal}, we have that $\|\bxs - \hat{\bx}\| \leq \frac{2 \delta}{\epsilon}$ with probability at least $1-\exp\left(-\tfrac{1}{2}\left[\lambda_n - w(\Omega) -\sqrt{n}\epsilon\right]^2\right)$ provided
	\begin{equation*}
	n \geq \frac{w(\Omega)^2+3/2}{(1-\epsilon)^2} \,.
	\end{equation*}
\end{enumerate}
\end{corollary}

\begin{proof}
The two results are simple consequences of Theorem~\ref{theo:escape} and a concentration of measure argument.  Recall that for an function $f:\R^d \rightarrow \R$ with Lipschitz constant $L$ and a random Gaussian vector, $\bg\in\R^d$, with mean zero and identity variance
\begin{equation}\label{eq:conc-of-meas}
	\P\left[ f(\bg) \geq \E[f]-t \right] \geq 1 -  \exp\left(-\frac{t^2}{2L^2} \right)
\end{equation}
(see, for example,~\cite{LedT1991,Pis1986}). For any set $\Omega \subset \mathbb{S}^{p-1}$, the function
\[
\Phi \mapsto \min_{\bz\in \Omega} \|\Phi \bz\|_2
\]
is Lipschitz with respect to the Frobenius norm with constant $1$.   Thus,  applying~\ref{theo:escape} and (\ref{eq:conc-of-meas}), we find that
\begin{equation}\label{eq:main-conc-ineq}
	\P\left[ \min_{\bz\in \Omega} \|\Phi \bz\|_2 \geq \epsilon \right] \geq 1 -  \exp\left(-\frac{1}{2}(\lambda_n - w(\Omega) - \sqrt{n}\epsilon)^2\right)
\end{equation}
provided that $\lambda_n - w(\Omega) - \sqrt{n}\epsilon\geq 0$.

The first part now follows by setting $\epsilon = 0$ in (\ref{eq:main-conc-ineq}).  The concentration inequality is valid provided that $\lambda_n \geq w(\Omega)$.  To verify this, note
\[
\lambda_n \geq \frac{n}{\sqrt{n+1}} \geq \sqrt{\frac{w(\Omega)^2+1}{1+1/n} }
\geq \sqrt{\frac{w(\Omega)^2+w(\Omega)^2/n}{1+1/n} } = w(\Omega)\,.
\]
Here, both inequalities use the fact that $n\geq w(\Omega)^2+1$.

For the second part, we have from (\ref{eq:main-conc-ineq}) that
    \[
    \|\Phi(\bz)\| = \|\bz\| \left\|\Phi\left(\frac{\bz}{\|\bz\|}\right)\right\| \geq \epsilon \|\bz\|
    \]
for all $\bz\in \cT_\A(\bxs)$ with high probability if $\lambda_n \geq w(\Omega)+\sqrt{n}\epsilon$.  In this case, we can apply Proposition~\ref{prop:noisy-recovery} to conclude that $\|\hat{\bx}-\bxs\|\leq \frac{2\delta}{\epsilon}$.  To verify that concentration of measure can be applied is more or less the same as in the proof of Part 1. First, note that under the assumptions of the theorem
\[
	w(\Omega)^2 +1 \leq n(1-\epsilon)^2 -1/2 \leq n(1-\epsilon)^2 -2\epsilon(1-\epsilon) +\frac{\epsilon^2}{n} = \left(\sqrt{n}(1-\epsilon) - \frac{\epsilon}{\sqrt{n}} \right)^2
\]
as $\epsilon(1-\epsilon)\leq 1/4$ for $\epsilon\in (0,1)$.  Using this fact, we then have
  \[
    \lambda_n - \sqrt{n} \epsilon \geq \frac{n- (n+1)\epsilon}{\sqrt{n+1}} \geq \sqrt{\frac{w(\Omega)^2+1}{1+1/n}} \geq w(\Omega)
    \]
    as desired.
\end{proof}

Gordon's theorem thus provides a simple characterization of the number of measurements required for reconstruction with the atomic norm.  Indeed the Gaussian width of $\Omega = T_\A(\bxs) \cap \Sp^{p-1}$ is the only quantity that we need to compute in order to obtain bounds for both exact and robust recovery.  Unfortunately it is in general not easy to compute Gaussian widths.  Rudelson and Vershynin \cite{RudV2006} have worked out Gaussian widths for the special case of tangent cones at sparse vectors on the boundary of the $\ell_1$ ball, and derived results for sparse vector recovery using $\ell_1$ minimization that improve upon previous results.  In the next section we give various well-known properties of the Gaussian width that are useful in computations.  In Section~\ref{subsec:newwidthprop} we discuss a new approach to width computations that gives near-optimal recovery bounds in a variety of settings.

\subsection{Properties of Gaussian Width}
\label{subsec:widthprop}

The Gaussian width has deep connections to convex geometry.  Since the length and direction of a Gaussian random vector are independent, one can verify that for $S \subset \R^p$
\[
	w(S) = \frac{\lambda_p}{2} \int_{\Sp^{p-1}} \left(\max_{\bz\in S} \bu^T \bz -\min_{\bz\in S} \bu^T \bz\right)\,d\bu =  \frac{\lambda_p}{2} \, b(S)
\]
where the integral is with respect to Haar measure on $\Sp^{p-1}$ and $b(S)$ is known as the \emph{mean width} of $S$.   The mean width measures the average length of $S$ along unit directions in $\R^p$ and is one of the fundamental \emph{intrinsic volumes} of a body studied in combinatorial geometry~\cite{KlaR1997}.  Any continuous valuation that is invariant under rigid motions and homogeneous of degree 1 is a multiple of the mean width and hence a multiple of the Gaussian width.  We can use this connection with convex geometry to underscore several properties of the Gaussian width that are useful for computation.

The Gaussian width of a body is invariant under translations and unitary transformations.  Moreover, it is homogeneous in the sense that $w(tK)$ = $tw(K)$ for $t>0$. The width is also monotonic.  If $S_1 \subseteq S_2 \subseteq \R^p$, then it is clear from the definition of the Gaussian width that
\begin{equation*}
w(S_1) \leq w(S_2).
\end{equation*}
Less obvious, the width is modular in the sense that if $S_1$ and $S_2$ are convex bodies with $S_1\cup S_2$ convex, we also have
\[
	w(S_1 \cup S_2) + w(S_1 \cap S_2) = w(S_1)+w(S_2)\,.
\]
This equality follows from the fact that $w$ is a valuation~\cite{Bar2002}. Also note that if we have a set $S \subseteq \R^p$, then the Gaussian width of $S$ is equal to the Gaussian width of the convex hull of $S$:
\begin{equation*}
w(S) = w(\mathrm{conv}(S)).
\end{equation*}
This result follows from the basic fact in convex analysis that the maximum of a convex function over a convex set is achieved at an extreme point of the convex set.

If $V \subset \R^p$ is a subspace in $\R^p$, then we have that
\begin{equation*}
w(V \cap \Sp^{p-1}) = \sqrt{\mathrm{dim}(V)},
\end{equation*}
which follows from standard results on random Gaussians.  This result also agrees with the intuition that a random Gaussian map $\Phi$ misses a $k$-dimensional subspace with high probability as long as $\mathrm{dim}(\n(\Phi)) \geq k + 1$.  Finally, if a cone $S \subset \R^p$ is such that $S = S_1 \oplus S_2$, where $S_1 \subset \R^p$ is a $k$-dimensional cone, $S_2 \subset \R^p$ is a $(p-k)$-dimensional cone that is orthogonal to $S_1$, and $\oplus$ denotes the direct sum operation, then the width can be decomposed as follows:
\begin{equation*}
w(S \cap \Sp^{p-1})^2 \leq w(S_1 \cap \Sp^{p-1})^2 + w(S_2 \cap \Sp^{p-1})^2.
\end{equation*}
These observations are useful in a variety of situations.  For example a width computation that frequently arises is one in which $S = S_1 \oplus S_2$ as described above, with $S_1$ being a $k$-dimensional subspace.  It follows that the width of $S \cap \Sp^{p-1}$ is bounded as
\begin{equation}
w(S \cap \Sp^{p-1})^2 \leq k + w(S_2 \cap \Sp^{p-1})^2. \label{eq:width-sub-cone}
\end{equation}

Another tool for computing Gaussian widths is based on Dudley's inequality \cite{Dud1967,LedT1991}, which bounds the width of a set in terms of the covering number of the set at all scales.

\begin{definition}
Let $S$ be an arbitrary compact subset of $\mathbb{R}^p$.  The \emph{covering number} of $S$ in the Euclidean norm at resolution $\epsilon$ is the smallest number, $\mathfrak{N}(S,\epsilon)$, such that $\mathfrak{N}(S,\epsilon)$ Euclidean balls of radius $\epsilon$ cover $S$.
\end{definition}

\begin{theorem}[Dudley's Inequality]
\label{theo:dudley}

Let $S$ be an arbitrary compact subset of $\R^p$, and let $\bg$ be a random vector with i.i.d. zero-mean, unit-variance Gaussian entries.  Then
\begin{equation}
        w(S) \leq 24 \int_0^\infty \sqrt{ \log(\mathfrak{N}(S,\epsilon))} d\epsilon.
\end{equation}
\end{theorem}

We note here that a weak converse to Dudley's inequality can be obtained via Sudakov's Minoration \cite{LedT1991} by using the covering number for just a single scale.  Specifically, we have the following \emph{lower bound} on the Gaussian width of a compact subset $S \subset \R^p$ for any $\epsilon > 0$:
\begin{equation*}
w(S) \geq c \epsilon \sqrt{ \log(\mathfrak{N}(S,\epsilon))}.
\end{equation*}
Here $c > 0$ is some universal constant.

Although Dudley's inequality can be applied quite generally, estimating covering numbers is difficult in most instances.  There are a few simple characterizations available for spheres and Sobolev spaces, and some tractable arguments based on Maurey's empirical method \cite{LedT1991}.  However it is not evident how to compute these numbers for general convex cones.  Also, in order to apply Dudley's inequality we need to estimate the covering number at all scales.  Further Dudley's inequality can be quite loose in its estimates, and it often introduces extraneous polylogarithmic factors.  In the next section we describe a new mechanism for estimating Gaussian widths, which provides near-optimal guarantees for recovery of sparse vectors and low-rank matrices, as well as for several of the recovery problems discussed in Section~\ref{subsec:newrec}.

\subsection{New Results on Gaussian Width}
\label{subsec:newwidthprop}

We now present a framework for computing Gaussian widths by bounding the Gaussian width of a cone via the distance to the dual cone.  To be fully general let $\mathcal{C}$ be a non-empty convex cone in $\R^p$, and let $\mathcal{C}^\ast$ denote the polar of $\mathcal{C}$.  We can then upper bound the Gaussian width of any cone $\mathcal{C}$ in terms of the polar cone $\mathcal{C}^\ast$:

\begin{proposition} \label{prop:dual-width}
Let $\mathcal{C}$ be any non-empty convex cone in $\R^p$, and let $\bg \sim \mathcal{N}(0,I)$ be a random Gaussian vector.  Then we have the following bound:
\begin{equation*}
w(\mathcal{C} \cap \Sp^{p-1}) \leq \E_\bg \left[\mathrm{dist}(\bg, \mathcal{C}^\ast) \right],
\end{equation*}
where $\mathrm{dist}$ here denotes the Euclidean distance between a point and a set.
\end{proposition}

The proof is given in Appendix~\ref{app:dual-width}, and it follows from an appeal to convex duality.  Proposition~\ref{prop:dual-width} is more or less a restatement of the fact that the support function of a convex cone is equal to the distance to its polar cone.  As it is the square of the Gaussian width that is of interest to us (see Corollary~\ref{corl:width}), it is often useful to apply Jensen's inequality to make the following approximation:
\begin{equation}\label{eq:square-jensen-bound}
\E_\bg[\mathrm{dist}(\bg,\mathcal{C}^\ast)]^2 \leq \E_\bg[\mathrm{dist}(\bg,\mathcal{C}^\ast)^2].
\end{equation}


The inspiration for our characterization in Proposition~\ref{prop:dual-width} of the width of a cone in terms of the expected distance to its dual came from the work of Stojnic \cite{Sto2009}, who used linear programming duality to construct Gaussian-width-based estimates for analyzing recovery in sparse reconstruction problems.  Specifically, Stojnic's relatively simple approach recovered well-known phase transitions in sparse signal recovery \cite{DonT2005}, and also generalized to block sparse signals and other forms of structured sparsity.

This new dual characterization yields a number of useful bounds on the Gaussian width, which we describe here.  In the following section we use these bounds to derive new recovery results.  The first result is a bound on the Gaussian width of a cone in terms of the Gaussian width of its polar.

\begin{lemma} \label{lemm:dual-width}
Let $\mathcal{C} \subseteq \R^p$ be a non-empty closed, convex cone.  Then we have that
\begin{equation*}
w(\mathcal{C} \cap \Sp^{p-1})^2 + w(\mathcal{C}^\ast \cap \Sp^{p-1})^2 \leq p.
\end{equation*}
\end{lemma}

\begin{proof}
Combining Proposition~\ref{prop:dual-width} and \eqref{eq:square-jensen-bound}, we have that
\begin{equation*}
w(\mathcal{C} \cap \Sp^{p-1})^2 \leq \E_\bg\left[\mathrm{dist}(\bg,\mathcal{C}^\ast)^2 \right],
\end{equation*}
where as before $\bg \sim \mathcal{N}(0,I)$.  For any $\bz \in \R^p$ we let $\Pi_\mathcal{C}(\bz) = \arg \inf_{\bu \in \mathcal{C}} \|\bz - \bu \|$ denote the projection of $\bz$ onto $\mathcal{C}$.  From standard results in convex analysis \cite{Roc1996}, we note that one can decompose any $\bz \in \R^p$ into orthogonal components as follows:
\begin{equation*}
\bz = \Pi_\mathcal{C}(\bz) + \Pi_{\mathcal{C}^\ast}(\bz), ~~~ \langle \Pi_\mathcal{C}(\bz), \Pi_{\mathcal{C}^\ast}(\bz) \rangle = 0.
\end{equation*}
Therefore we have the following sequence of bounds:
\begin{eqnarray*}
w(\mathcal{C} \cap \Sp^{p-1})^2 &\leq& \E_\bg\left[\mathrm{dist}(\bg,\mathcal{C}^\ast)^2 \right] \\ &=& \E_\bg\left[ \|\Pi_{\mathcal{C}}(\bg)\|^2\right] \\ &=& \E_\bg\left[ \|\bg\|^2 - \|\Pi_{\mathcal{C}^\ast}(\bg)\|^2\right] \\ &=& p - \E_\bg \left[\|\Pi_{\mathcal{C}^\ast}(\bg)\|^2 \right] \\ &=& p - \E_\bg \left[ \mathrm{dist}(\bg, \mathcal{C})^2 \right] \\ &\leq& p - w(\mathcal{C}^\ast \cap \Sp^{p-1})^2.
\end{eqnarray*}
\end{proof}

In many recovery problems one is interested in computing the width of a self-dual cone.  For such cones the following corollary to Lemma~\ref{lemm:dual-width} gives a simple solution:

\begin{corollary} \label{corl:selfdual}
Let $\mathcal{C} \subset \R^p$ be a self-dual cone, i.e., $\mathcal{C} = -\mathcal{C}^\ast$.  Then we have that
\begin{equation*}
w(\mathcal{C} \cap \Sp^{p-1})^2 \leq \frac{p}{2}.
\end{equation*}
\end{corollary}

\begin{proof}
The proof follows directly from Lemma~\ref{lemm:dual-width} as $w(\mathcal{C} \cap \Sp^{p-1})^2 = w(\mathcal{C}^\ast \cap \Sp^{p-1})^2$.
\end{proof}

Our next bound for the width of a cone $\mathcal{C}$ is based on the volume of its polar $\mathcal{C}^\ast \cap \Sp^{p-1}$.  The \emph{volume} of a measurable subset of the sphere is the fraction of the sphere $\Sp^{p-1}$ covered by the subset.  Thus it is a quantity between zero and one.

\begin{theorem} [Gaussian width from volume of the polar]
\label{theo:angle}

Let $\mathcal{C} \subseteq \R^p$ be any closed, convex, solid cone, and suppose that its polar $\mathcal{C}^\ast$ is such that $\mathcal{C}^\ast \cap \Sp^{p-1}$ has a volume of $\Theta \in [0, 1]$.  Then for $p \geq 9$ we have that
\begin{equation*}
w(\mathcal{C} \cap \Sp^{p-1}) \leq 3 \sqrt{\log\left(\frac{4}{\Theta}\right)}.
\end{equation*}
\end{theorem}

The proof of this theorem is given in Appendix~\ref{app:width-angle}.  The main property that we appeal to in the proof is \emph{Gaussian isoperimetry}.  In particular there is a formal sense in which a spherical cap\footnote{A \emph{spherical cap} is a subset of the sphere obtained by intersecting the sphere $\Sp^{p-1}$ with a halfspace.} is the ``extremal case'' among all subsets of the sphere with a given volume $\Theta$.  Other than this observation the proof mainly involves a sequence of integral calculations.

Note that if we are given a specification of a cone $\mathcal{C} \subset \R^p$ in terms of a membership oracle, it is possible to efficiently obtain good numerical estimates of the volume of $\mathcal{C} \cap \Sp^{p-1}$ \cite{DyeFK1991}.  Moreover, simple symmetry arguments often give relatively accurate estimates of these volumes. Such estimates can then be plugged into Theorem~\ref{theo:angle} to yield bounds on the width.

\subsection{New Recovery Bounds}
\label{subsec:newrec}

We use the bounds derived in the last section to obtain new recovery results.  First using the dual characterization of the Gaussian width in Proposition~\ref{prop:dual-width}, we are able to obtain sharp bounds on the number of measurements required for recovering sparse vectors and low-rank matrices from random Gaussian measurements using convex optimization (i.e., $\ell_1$-norm and nuclear norm minimization).

\begin{proposition}\label{prop:l1}
Let $\bxs \in \R^p$ be an $s$-sparse vector.  Letting $\A$ denote the set of unit-Euclidean-norm one-sparse vectors, we have that
\begin{equation*}
w(T_{\A}(\bxs) \cap \Sp^{p-1})^2 \leq 2s \log\left(\tfrac{p}{s}\right) + \tfrac{5}{4}s\,.
\end{equation*}
Thus, $2s \log\left(\tfrac{p}{s}\right) + \tfrac{5}{4}s+1$ random Gaussian measurements suffice to recover $\bxs$ via $\ell_1$ norm minimization with high probability.
\end{proposition}

\begin{proposition}\label{prop:nuclear}
Let $\bxs$ be an $m_1 \times m_2$ rank-$r$ matrix with $m_1 \leq m_2$.  Letting $\A$ denote the set of unit-Euclidean-norm rank-one matrices, we have that
\begin{equation*}
w\left(T_{\A}(\bxs) \cap \Sp^{m_1 m_2-1}\right)^2 \leq 3 r(m_1+m_2-r).
\end{equation*}
Thus $3 r(m_1+m_2-r)+1$ random Gaussian measurements suffice to recover $\bxs$ via nuclear norm minimization with high probability.
\end{proposition}

The proofs of these propositions are given in Appendix~\ref{app:direct}.  The number of measurements required by these bounds is on the same order as previously known results. In the case of sparse vectors, previous results getting $2s\log (p/s)$ were asymptotic \cite{DonT2009}.  Our bounds, in contrast, hold with high probability in finite dimensions.  In the case of low-rank matrices, our bound provides considerably sharper constants than those previously derived (as in, for example~\cite{CanP2009}).  We also note that we have robust recovery at these thresholds.  Further these results do not require explicit recourse to any type of restricted isometry property \cite{CanP2009}, and the proofs are simple and based on elementary integrals.

Next we obtain a set of recovery results by appealing to Corollary~\ref{corl:selfdual} on the width of a self-dual cone.  These examples correspond to the recovery of individual atoms (i.e., the extreme points of the set $\ch(\A)$), although the same machinery is applicable in principle to estimate the number of measurements required to recover models formed as sums of a few atoms (i.e., points lying on low-dimensional faces of $\ch(\A)$).  We first obtain a well-known result on the number of measurements required for recovering sign-vectors via $\ell_\infty$ norm minimization.

\begin{proposition} \label{prop:sign}
Let $\A \in \{-1,+1\}^p$ be the set of sign-vectors in $\R^p$.  Suppose $\bxs \in \R^p$ is a vector formed as a convex combination of $k$ sign-vectors in $\A$ such that $\bxs$ lies on a $k$-face of the $\ell_\infty$-norm unit ball.  Then we have that
\begin{equation*}
w(T_{\A}(\bxs) \cap \Sp^{p-1})^2 \leq \frac{p+k}{2}.
\end{equation*}
Thus $\tfrac{p+k}{2}$ random Gaussian measurements suffice to recover $\bxs$ via $\ell_\infty$-norm minimization with high probability.
\end{proposition}

\begin{proof}
The tangent cone at $\bxs$ with respect to the $\ell_\infty$-norm ball is the direct sum of a $k$-dimensional subspace and a (rotated) $(p-k)$-dimensional nonnegative orthant.  As the orthant is self-dual, we obtain the required bound by combining Corollary~\ref{corl:selfdual} and \eqref{eq:width-sub-cone}.
\end{proof}

This result agrees with previously computed bounds in \cite{ManR2009,DonT2010}, which relied on a more complicated combinatorial argument.  Next we compute the number of measurements required to recover orthogonal matrices via spectral-norm minimization (see Section~\ref{subsec:ex}).  Let $\mathbb{O}(m)$ denote the group of $m \times m$ orthogonal matrices, viewed as a subgroup of the set of nonsingular matrices in $\R^{m \times m}$.

\begin{proposition} \label{prop:ortho}
Let $\bxs \in \R^{m \times m}$ be an orthogonal matrix, and let $\A$ be the set of all orthogonal matrices.  Then we have that
\begin{equation*}
w(T_{\A}(\bxs) \cap \Sp^{m^2-1})^2 \leq \frac{3 m^2 - m}{4}.
\end{equation*}
Thus $\tfrac{3 m^2 - m}{4}$ random Gaussian measurements suffice to recover $\bxs$ via spectral-norm minimization with high probability.
\end{proposition}

\begin{proof}
Due to the symmetry of the orthogonal group, it suffices to consider the tangent cone at the identity matrix $I$ with respect to the spectral norm ball.  Recall that the spectral norm ball is the convex hull of the orthogonal matrices.  Therefore the \emph{tangent space} at the identity matrix with respect to the orthogonal group $\mathbb{O}(m)$ is a subset of the tangent cone $T_\A(I)$.  It is well-known that this tangent space is the Lie Algebra of all $m \times m$ skew-symmetric matrices.  Thus we only need to compute the component $S$ of $T_\A(I)$ that lies in the subspace of symmetric matrices:
\begin{eqnarray*}
S &=& \mathrm{cone}\{M - I : \|M\|_\A \leq 1, ~ M \mathrm{~ symmetric} \} \\ &=& \mathrm{cone}\{U D U^T - U U^T : \|D\|_\A \leq 1, ~ D \mathrm{~ diagonal}, ~ U \in \mathbb{O}(m) \} \\ &=& \mathrm{cone}\{U (D-I) U^T: \|D\|_\A \leq 1, ~ D \mathrm{~ diagonal}, ~ U \in \mathbb{O}(m) \} \\ &=& -\mathrm{PSD}_m.
\end{eqnarray*}
Here $\mathrm{PSD}_m$ denotes the set of $m \times m$ symmetric positive-semidefinite matrices.  As this cone is self-dual, we can apply Corollary~\ref{corl:selfdual} in conjunction with the observations in Section~\ref{subsec:widthprop} to conclude that
\begin{equation*}
w(T_\A(I) \cap \Sp^{m^2-1})^2 \leq {m \choose 2} + \frac{1}{2}{m+1 \choose 2} = \frac{3m^2 - m}{4}.
\end{equation*}
\end{proof}

We note that the number of degrees of freedom in an $m \times m$ orthogonal matrix (i.e., the dimension of the manifold of orthogonal matrices) is $\tfrac{m(m-1)}{2}$.  Proposition~\ref{prop:sign} and Proposition~\ref{prop:ortho} point to the importance of obtaining recovery bounds with sharp constants.  Larger constants in either result would imply that the number of measurements required exceeds the ambient dimension of the underlying $\bxs$.  In these and many other cases of interest Gaussian width arguments not only give order-optimal recovery results, but also provide precise constants that result in sharp recovery thresholds.

Finally we give a third set of recovery results that appeal to the Gaussian width bound of Theorem~\ref{theo:angle}.  The following measurement bound applies to cases when $\ch(\A)$ is a \emph{symmetric polytope} (roughly speaking, all the vertices are ``equivalent''), and is a simple corollary of Theorem~\ref{theo:angle}.

\begin{corollary} \label{corl:symm}
Suppose that the set $\A$ is a finite collection of $m$ points, with the convex hull $\ch(\A)$ being a vertex-transitive polytope \cite{Zie1995} whose vertices are the points in $\A$.  Using the convex program \eqref{eq:atomic-norm-primal} we have that $9 \log(m)$ random Gaussian measurements suffice, with high probability, for exact recovery of a point in $\A$, i.e., a vertex of $\ch(\A)$.
\end{corollary}

\begin{proof}
We recall the basic fact from convex analysis that the normal cones at the vertices of a convex polytope in $\R^p$ provide a partitioning of $\R^p$.  As $\ch(\A)$ is a vertex-transitive polytope, the normal cone at a vertex covers $\tfrac{1}{m}$ fraction of $\R^p$.  Applying Theorem~\ref{theo:angle}, we have the desired result.
\end{proof}

Clearly we require the number of vertices to be bounded as $m \leq \exp\{\tfrac{p}{9}\}$, so that the estimate of the number of measurements is not vacuously true.  This result has useful consequences in settings in which $\ch(\A)$ is a \emph{combinatorial polytope}, as such polytopes are often vertex-transitive.  We have the following example on the number of measurements required to recover permutation matrices\footnote{While Proposition~\ref{prop:birkhoff} follows as a consequence of the general result in Corollary~\ref{corl:symm}, one can remove the constant factor $9$ in the statement of Proposition~\ref{prop:birkhoff} by carrying out a more refined analysis of the Birkhoff polytope.}:

\begin{proposition} \label{prop:birkhoff}
Let $\bxs \in \R^{m \times m}$ be a permutation matrix, and let $\A$ be the set of all $m \times m$ permutation matrices.  Then $9 \, m \log(m)$ random Gaussian measurements suffice, with high probability, to recover $\bxs$ by solving the optimization problem \eqref{eq:atomic-norm-primal}, which minimizes the norm induced by the Birkhoff polytope of doubly stochastic matrices.
\end{proposition}

\begin{proof}
This result follows from Corollary~\ref{corl:symm} by noting that there are $m!$ permutation matrices of size $m \times m$.
\end{proof}



\section{Representability and Algebraic Geometry of Atomic Norms}
\label{sec:rep}


All of our discussion thus far has focussed on arbitrary atomic sets $\A$.  As seen in Section~\ref{sec:def} the geometry of the convex hull $\ch(\A)$ completely determines conditions under which exact recovery is possible using the convex program \eqref{eq:atomic-norm-primal}.  In this section we address the question of computing atomic norms for general sets of atoms.  These issues are critical in order to be able to solve the convex optimization problem \eqref{eq:atomic-norm-primal}.  Although the convex hull $\ch(\A)$ is always a mathematically well-defined object, testing membership in this set is in general undecidable (for example, if $\A$ is a fractal).  Further, even if these convex hulls are computable they may not admit efficient representations.  For example if $\A$ is the set of rank-one signed matrices (see Section~\ref{subsec:ex}), the corresponding convex hull $\ch(\A)$ is the cut polytope for which there is no known tractable characterization.  Consequently, one may have to resort to efficiently computable approximations of $\ch(\A)$. The tradeoff in using such approximations in our atomic norm minimization framework is that we require more measurements for robust recovery.  This section is devoted to providing a better understanding of these issues.

\subsection{Role of Algebraic Structure}
\label{subsec:algst}

In order to obtain exact or approximate representations (analogous to the cases of the $\ell_1$ norm and the nuclear norm) it is important to identify properties of the atomic set $\A$ that can be exploited computationally.  We focus on cases in which the set $\A$ has algebraic structure.  Specifically let the ring of multivariate polynomials in $p$ variables be denoted by $\R[\bx] = \R[\bx_1,\dots,\bx_p]$.  We then consider real algebraic varieties \cite{BocCR1998}:

\begin{definition}
A \emph{real algebraic variety} $S \subseteq \R^p$ is the set of real solutions of a system of polynomial equations:
\begin{equation*}
S = \{\bx : g_j(\bx) = 0, ~ \forall j\},
\end{equation*}
where $\{g_j\}$ is a finite collection of polynomials in $\R[\bx]$.
\end{definition}

Indeed all of the atomic sets $\A$ considered in this paper are examples of algebraic varieties.  Algebraic varieties have the remarkable property that (the closure of) their convex hull can be arbitrarily well-approximated in a constructive manner as (the projection of) a set defined by linear matrix inequality constraints \cite{GouPT2010,Par2003}.  A potential complication may arise, however, if these semidefinite representations are intractable to compute in polynomial time.  In such cases it is possible to approximate the convex hulls via a hierarchy of tractable semidefinite relaxations.  We describe these results in more detail in Section~\ref{subsec:psatz}.  Therefore the atomic norm minimization problems such as \eqref{eq:noisy-atomic-norm-primal} arising in such situations can be solved exactly or approximately via semidefinite programming.

Algebraic structure also plays a second important role in atomic norm minimization problems.  If an atomic norm $\| \cdot \|_\A$ is intractable to compute, we may approximate it via a more tractable norm $\|\cdot\|_{app}$.  However not every approximation of the atomic norm is equally good for solving inverse problems. As illustrated in Figure~\ref{fig:fig2} we can construct approximations of the $\ell_1$ ball that are tight in a \emph{metric} sense, with $(1-\epsilon) \| \cdot \|_{app} \leq \|\cdot\|_{\ell_1} \leq (1+\epsilon) \|\cdot\|_{app}$, but where the tangent cones at sparse vectors in the new norm are halfspaces. In such a case, the number of measurements required to recover the sparse vector ends up being on the same order as the ambient dimension.  (Note that the $\ell_1$-norm is in fact tractable to compute; we simply use it here for illustrative purposes.)  The key property that we seek in approximations to an atomic norm $\|\cdot\|_\A$ is that they \emph{preserve algebraic structure} such as the vertices/extreme points and more generally the low-dimensional faces of the $\ch(\A)$.  As discussed in Section~\ref{subsec:why} points on such low-dimensional faces correspond to simple models, and algebraic-structure preserving approximations ensure that the tangent cones at simple models with respect to the approximations are not too much larger than the corresponding tangent cones with respect to the original atomic norms (see Section~\ref{subsec:tradeoff} for a concrete example).

\begin{figure}
\begin{center}
\epsfig{file=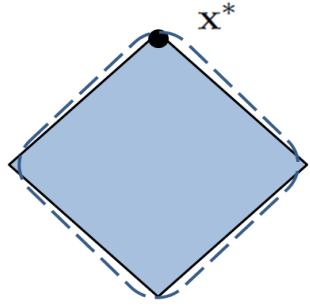,width=4cm,height=4cm} \caption{The convex body given by the dotted line is a good metric approximation to the $\ell_1$ ball.  However as its ``corners'' are ``smoothed out'', the tangent cone at $\bxs$ goes from being a proper cone (with respect to the $\ell_1$ ball) to a halfspace (with respect to the approximation).} \label{fig:fig2}
\end{center}
\end{figure}

\subsection{Semidefinite Relaxations using Theta Bodies}
\label{subsec:psatz}

In this section we give a family of semidefinite relaxations to the atomic norm minimization problem whenever the atomic set has algebraic structure.  To begin with if we approximate the atomic norm $\|\cdot\|_\A$ by another atomic norm $\|\cdot\|_{\tilde{\A}}$ defined using a \emph{larger} collection of atoms $\A \subseteq \tilde{\A}$, it is clear that
\begin{equation*}
\| \cdot \|_{\tilde{\A}} \leq \|\cdot\|_\A.
\end{equation*}
Consequently outer approximations of the atomic set give rise to approximate norms that provide lower bounds on the optimal value of the problem \eqref{eq:atomic-norm-primal}.

In order to provide such lower bounds on the optimal value of \eqref{eq:atomic-norm-primal}, we discuss semidefinite relaxations of the convex hull $\ch(\A)$.  All our discussion here is based on results described in \cite{GouPT2010} for semidefinite relaxations of convex hulls of algebraic varieties using theta bodies.  We only give a brief review of the relevant constructions, and refer the reader to the vast literature on this subject for more details (see \cite{GouPT2010,Par2003} and the references therein).  For subsequent reference in this section, we recall the definition of a polynomial ideal \cite{BocCR1998,Har95}:

\begin{definition}
A \emph{polynomial ideal} $I \subset \R[\bx]$ is a subset of the ring of polynomials that contains the zero polynomial (the polynomial that is identically zero), is closed under addition, and has the property that $f \in I, g \in \R[\bx]$ implies that $f \cdot g \in I$.
\end{definition}

To begin with we note that a \emph{sum-of-squares} (SOS) polynomial in $\R[\bx]$ is a polynomial that can be written as the (finite) sum of squares of other polynomials in $\R[\bx]$.  Verifying the nonnegativity of a multivariate polynomial is intractable in general, and therefore SOS polynomials play an important role in real algebraic geometry as an SOS polynomial is easily seen to be nonnegative everywhere.  Further checking whether a polynomial is an SOS polynomial can be accomplished efficiently via semidefinite programming \cite{Par2003}.

Turning our attention to the description of the convex hull of an algebraic variety, we will assume for the sake of simplicity that the convex hull is closed.  Let $I \subseteq \R[\bx]$ be a polynomial ideal, and let $V_\R(I) \in \R^p$ be its real algebraic variety:
\begin{equation*}
V_\R(I) = \{\bx : f(\bx) = 0, ~ \forall f \in I\}.
\end{equation*}
One can then show that the convex hull $\ch(V_\R(I))$ is given as:
\begin{eqnarray*}
\ch(V_\R(I)) &=& \{\bx : f(\bx) \geq 0, ~ \forall f ~\mathrm{linear ~ and ~ nonnegative ~ on ~} V_\R(I)\} \\ &=& \{\bx : f(\bx) \geq 0, ~ \forall f ~\mathrm{linear ~ s.t. ~} f = h + g, ~\forall~ h ~\mathrm{nonnegative, ~} \forall ~ g \in I\} \\ &=& \{\bx : f(\bx) \geq 0, ~ \forall f ~\mathrm{linear ~ s.t. ~} f ~\mathrm{nonnegative ~ modulo~} I\}.
\end{eqnarray*}
A linear polynomial here is one that has a maximum degree of one, and the meaning of ``modulo an ideal'' is clear.  As nonnegativity modulo an ideal may be intractable to check, we can consider a relaxation to a polynomial being SOS modulo an ideal, i.e., a polynomial that can be written as $\sum_{i=1}^q ~ h_i^2 ~ + ~ g$ for $g$ in the ideal.  Since it is tractable to check via semidefinite programmming whether bounded-degree polynomials are SOS, the $k$-th theta body of an ideal $I$ is defined as follows in \cite{GouPT2010}:
\begin{equation*}
\mathrm{TH}_k(I) = \{\bx : f(\bx) \geq 0, ~ \forall f ~\mathrm{linear ~ s.t. ~} f ~\mathrm{is~}k\mathrm{\mbox{-}sos ~ modulo~} I\}.
\end{equation*}
Here $k$-sos refers to an SOS polynomial in which the components in the SOS decomposition have degree at most $k$.  The $k$-th theta body $\mathrm{TH}_k(I)$ is a convex relaxation of $\ch(V_\R(I))$, and one can verify that
\begin{equation*}
\ch(V_\R(I)) \subseteq \cdots \subseteq \mathrm{TH}_{k+1}(I) \subseteq \mathrm{TH}_{k}(V_\R(I)).
\end{equation*}
By the arguments given above (see also \cite{GouPT2010}) these theta bodies can be described using semidefinite programs of size polynomial in $k$.  Hence by considering theta bodies $\mathrm{TH}_k(I)$ with increasingly larger $k$, one can obtain a hierarchy of tighter semidefinite relaxations of $\ch(V_\R(I))$.  We also note that in many cases of interest such semidefinite relaxations preserve low-dimensional faces of the convex hull of a variety, although these properties are not known in general.  We will use some of these properties below when discussing approximations of the cut-polytope.

\textbf{Approximating tensor norms.} We conclude this section with an example application of these relaxations to the problem of approximating the tensor nuclear norm.  We focus on the case of tensors of order three that lie in $\R^{m \times m \times m}$, i.e., tensors indexed by three numbers, for notational simplicity, although our discussion is applicable more generally.  In particular the atomic set $\A$ is the set of unit-Euclidean-norm rank-one tensors:
\begin{eqnarray*}
\A &=& \{\bu \otimes \bv \otimes \bw : \bu, \bv, \bw \in \R^m, ~ \|\bu\| = \|\bv\| = \|\bw\| = 1\} \\ &=& \{N \in \R^{m^3} : N = \bu \otimes \bv \otimes \bw, ~ \bu, \bv, \bw \in \R^m, ~ \|\bu\| = \|\bv\| = \|\bw\| = 1\},
\end{eqnarray*}
where $\bu \otimes \bv \otimes \bw$ is the tensor product of three vectors.  Note that the second description is written as the projection onto $\R^{m^3}$ of a variety defined in $\R^{m^3+3m}$.  The nuclear norm is then given by \eqref{eq:atnorm}, and is intractable to compute in general.  Now let $I_\A$ denote a polynomial ideal of polynomial maps from $\R^{m^3+3m}$ to $\R$:
\begin{equation*}
I_\A = \{g: g = \sum_{i,j,k=1}^m g_{ijk} (N_{ijk} - \bu_i \bv_j \bw_k)+g_u (\bu^T \bu-1) + g_v (\bv^T \bv-1) + g_w (\bw^T \bw-1), \forall g_{ijk},g_u,g_v,g_w\}.
\end{equation*}
Here $g_u,g_v,g_w,\{g_{ijk}\}_{i,j,k}$ are polynomials in the variables $N,\bu,\bv,\bw$.  Following the program described above for constructing approximations, a family of semidefinite relaxations to the tensor nuclear norm ball can be prescribed in this manner via the theta bodies $\mathrm{TH}_k(I_\A)$.

\subsection{Tradeoff between Relaxation and Number of Measurements}
\label{subsec:tradeoff}

As discussed in Section~\ref{subsec:why} the atomic norm is the best convex heuristic for solving ill-posed linear inverse problems of the type considered in this paper.  However we may wish to approximate the atomic norm in cases when it is intractable to compute exactly, and the discussion in the preceding section provides one approach to constructing a family of relaxations. As one might expect the tradeoff for using such approximations, i.e., a \emph{weaker} convex heuristic than the atomic norm, is an increase in the number of measurements required for exact or robust recovery.  The reason for this is that the approximate norms have \emph{larger} tangent cones at their extreme points, which makes it harder to satisfy the empty intersection condition of Proposition~\ref{prop:null-intersection}.  We highlight this tradeoff here with an illustrative example involving the cut polytope.

The cut polytope is defined as the convex hull of all cut matrices:
\begin{equation*}
\mathcal{P} = \mathrm{conv}\{\bz \bz^T : \bz \in \{-1,+1\}^m\}.
\end{equation*}
As described in Section~\ref{subsec:ex} low-rank matrices that are composed of $\pm 1$'s as entries are of interest in collaborative filtering \cite{SreS2005}, and the norm induced by the cut polytope is a potential convex heuristic for recovering such matrices from limited measurements.  However it is well-known that the cut polytope is intractable to characterize \cite{DezL1997}, and therefore we need to use tractable relaxations instead.  We consider the following two relaxations of the cut polytope.  The first is the popular relaxation that is used in semidefinite approximations of the MAXCUT problem:
\begin{equation*}
\mathcal{P}_1 = \{M : M ~\mathrm{symmetric}, ~ M \succeq 0, ~ M_{ii} = 1,  \forall i = 1,\cdots,p \}.
\end{equation*}
This is the well-studied elliptope \cite{DezL1997}, and can be interpreted as the second theta body relaxation (see Section~\ref{subsec:psatz}) of the cut polytope $\mathcal{P}$ \cite{GouPT2010}.  We also investigate the performance of a second, weaker relaxation:
\begin{equation*}
\mathcal{P}_2 = \{M: M ~\mathrm{symmetric}, ~ M_{ii} = 1, \forall i, ~ |M_{ij}| \leq 1, \forall i \neq j \}.
\end{equation*}
This polytope is simply the convex hull of symmetric matrices with $\pm 1$'s in the off-diagonal entries, and $1$'s on the diagonal.  We note that $\mathcal{P}_2$ is an extremely weak relaxation of $\mathcal{P}$, but we use it here only for illustrative purposes.  It is easily seen that
\begin{equation*}
\mathcal{P} \subset \mathcal{P}_1 \subset \mathcal{P}_2,
\end{equation*}
with all the inclusions being strict.  Figure~\ref{fig:fig3} gives a toy sketch that highlights all the main geometric aspects of these relaxations.  In particular $\mathcal{P}_1$ has many more extreme points that $\mathcal{P}$, although the set of vertices of $\mathcal{P}_1$, i.e., points that have full-dimensional normal cones, are precisely the cut matrices (which are the vertices of $\mathcal{P}$) \cite{DezL1997}.  The convex polytope $\mathcal{P}_2$ contains many more vertices compared to $\mathcal{P}$ as shown in Figure~\ref{fig:fig3}.  As expected the tangent cones at vertices of $\mathcal{P}$ become increasingly larger as we use successively weaker relaxations.  The following result summarizes the number of random measurements required for recovering a cut matrix, i.e., a rank-one sign matrix, using the norms induced by each of these convex bodies.

\begin{figure}
\begin{center}
\epsfig{file=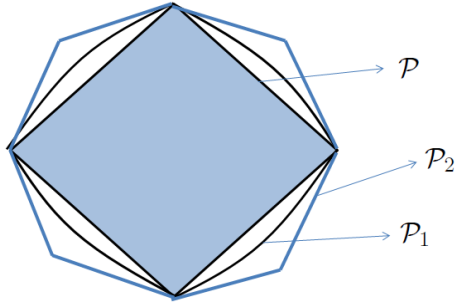,width=6cm,height=4cm} \caption{A toy sketch illustrating the cut polytope $\mathcal{P}$, and the two approximations $\mathcal{P}_1$ and $\mathcal{P}_2$. Note that $\mathcal{P}_1$ is a sketch of the standard semidefinite relaxation that has the \emph{same} vertices as $\mathcal{P}$.  On the other hand $\mathcal{P}_2$ is a polyhedral approximation to $\mathcal{P}$ that has many more vertices as shown in this sketch.} \label{fig:fig3}
\end{center}
\end{figure}

\begin{proposition} \label{prop:cut}
Suppose $\bxs \in \R^{m \times m}$ is a rank-one sign matrix, i.e., a cut matrix, and we are given $n$ random Gaussian measurements of $\bxs$.  We wish to recover $\bxs$ by solving a convex program based on the norms induced by each of $\mathcal{P}, \mathcal{P}_1, \mathcal{P}_2$.  We have exact recovery of $\bxs$ in each of these cases with high probability under the following conditions on the number of measurements:
\begin{enumerate}
\item Using $\mathcal{P}$: $n = \mathcal{O}(m)$.

\item Using $\mathcal{P}_1$: $n = \mathcal{O}(m)$.

\item Using $\mathcal{P}_2$: $n = \tfrac{m^2-m}{4}$.
\end{enumerate}
\end{proposition}

\begin{proof}
For the first part, we note that $\mathcal{P}$ is a symmetric polytope with $2^{m-1}$ vertices.  Therefore we can apply Corollary~\ref{corl:symm} to conclude that $n = \mathcal{O}(m)$ measurements suffices for exact recovery.

For the second part we note that the tangent cone at $\bxs$ with respect to the nuclear norm ball of $m \times m$ matrices contains within it the tangent cone at $\bxs$ with respect to the polytope $\mathcal{P}_1$. Hence we appeal to Proposition~\ref{prop:nuclear} to conclude that $n = \mathcal{O}(m)$ measurements suffices for exact recovery.

Finally, we note that $\mathcal{P}_2$ is essentially the hypercube in ${m \choose 2}$ dimensions.  Appealing to Proposition~\ref{prop:sign}, we conclude that $n = \tfrac{m^2-m}{4}$ measurements suffices for exact recovery.
\end{proof}

It is not too hard to show that these bounds are order-optimal, and that they cannot be improved.  Thus we have a rigorous demonstration in this particular instance of the fact that the number of measurements required for exact recovery increases as the relaxations get weaker (and as the tangent cones get larger).  The principle underlying this illustration holds more generally, namely that there exists a tradeoff between the complexity of the convex heuristic and the number of measurements required for exact or robust recovery.  It would be of interest to quantify this tradeoff in other settings, for example, in problems in which we use increasingly tighter relaxations of the atomic norm via theta bodies.

We also note that the tractable relaxation based on $\mathcal{P}_1$ is only off by a constant factor with respect to the optimal heuristic based on the cut polytope $\mathcal{P}$.  This suggests the potential for tractable heuristics to approximate hard atomic norms with provable approximation ratios, akin to methods developed in the literature on approximation algorithms for hard combinatorial optimization problems.

\subsection{Terracini's Lemma and Lower Bounds on Recovery}
\label{subsec:terracini}


Algebraic structure in the atomic set $\A$ also provides a means for computing \emph{lower bounds} on the number of measurements required for exact recovery.  The recovery condition of Proposition~\ref{prop:null-intersection} states that the nullspace $\n(\Phi)$ of the measurement operator $\Phi: \R^p \rightarrow \R^n$ must miss the tangent cone $T_\A(\bxs)$ at the point of interest $\bxs$.  Suppose that this tangent cone contains a $q$-dimensional subspace.  It is then clear from straightforward linear algebra arguments that the number of measurements $n$ must exceed $q$.  Indeed this bound must hold for \emph{any} linear measurement scheme.  Thus the dimension of the subspace contained inside the tangent cone (i.e., the dimension of the lineality space) provides a simple lower bound on the number of linear measurements.

In this section we discuss a method to obtain estimates of the dimension of a subspace component of the tangent cone.  We focus again on the setting in which $\A$ is an algebraic variety.  Indeed in all of the examples of Section~\ref{subsec:ex}, the atomic set $\A$ is an algebraic variety.  In such cases simple models $\bxs$ formed according to \eqref{eq:simp1} can be viewed as elements of \emph{secant varieties}.

\begin{definition}
Let $\A \in \R^p$ be an algebraic variety.  Then the $k$'th \emph{secant variety} $\A^k$ is defined as the union of all affine spaces passing through any $k+1$ points of $\A$.
\end{definition}

Secant varieties and their tangent spaces have been extensively studied in algebraic geometry \cite{Har95}.  A particular question of interest is to characterize the dimensions of secant varieties and tangent spaces.  In our context, estimates of these dimensions are useful in giving lower bounds on the number of measurements required for recovery.  Specifically we have the following result, which states that certain linear spaces must lie in the tangent cone at $\bxs$ with respect to $\ch(\A)$:

\begin{proposition} \label{prop:tspace}
Let $\A \subset \R^p$ be a smooth variety, and let $\mathcal{T}(\bu,\A)$ denote the tangent space at any $\bu \in \A$ with respect to $\A$.  Suppose $\bx = \sum_{i=1}^k c_i \ba_i, ~ \forall \ba_i \in \A, c_i \geq 0$, such that
\begin{equation*}
\|\bx\|_\A = \sum_{i=1}^k c_i.
\end{equation*}
Then the \emph{tangent cone} $T_\A(\bxs)$ contains the following linear space:
\begin{equation*}
\mathcal{T}(\ba_1,\A) \oplus \cdots \oplus \mathcal{T}(\ba_k,\A) \subset T_\A(\bxs),
\end{equation*}
where $\oplus$ denotes the direct sum of subspaces.
\end{proposition}

\begin{proof}
We note that if we perturb $\ba_1$ slightly to \emph{any} neighboring $\ba_1'$ so that $\ba_1' \in \A$, then the resulting $\bx' = c_1 \ba_1' + \sum_{i = 2}^k c_2 \ba_i$ is such that $\|\bx'\|_\A \leq \|\bx\|_\A$.  The proposition follows directly from this observation.
\end{proof}

This result is applicable, for example, when $\A$ is the variety of rank-one matrices or the variety of rank-one tensors as these are smooth varieties.  By Terracini's lemma \cite{Har95} from algebraic geometry the subspace $\mathcal{T}(\ba_1,\A) \oplus \cdots \oplus \mathcal{T}(\ba_k,\A)$ is in fact the estimate for the tangent space $\mathcal{T}(\bx,\A^{k-1})$ at $\bx$ with respect to the $(k-1)$'th secant variety $\A^{k-1}$:

\begin{proposition}[Terracini's Lemma] \label{prop:terracini}
Let $\A \subset \R^p$ be a smooth affine variety, and let $\mathcal{T}(\bu,\A)$ denote the tangent space at any $\bu \in \A$ with respect to $\A$.  Suppose $\bx \in \A^{k-1}$ is a generic point such that $\bx = \sum_{i=1}^k c_i \ba_i, ~ \forall \ba_i \in \A, c_i \geq 0$.  Then the tangent space $\mathcal{T}(\bx,\A^{k-1})$ at $\bx$ with respect to the secant variety $\A^{k-1}$ is given by $\mathcal{T}(\ba_1,\A) \oplus \cdots \oplus \mathcal{T}(\ba_k,\A)$.  Moreover the dimension of $\mathcal{T}(\bx,\A^{k-1})$ is at most (and is expected to be) $\min\{p, (k+1)\mathrm{dim}(\A) + k\}$.
\end{proposition}

Combining these results we have that estimates of the dimension of the tangent space $\mathcal{T}(\bx,\A^{k-1})$ lead directly to lower bounds on the number of measurements required for recovery.  The intuition here is clear as the number of measurements required must be bounded below by the number of ``degrees of freedom,'' which is captured by the dimension of the tangent space $\mathcal{T}(\bx,\A^{k-1})$.  However Terracini's lemma provides us with general estimates of the dimension of $\mathcal{T}(\bx,\A^{k-1})$ for generic points $\bx$.  Therefore we can directly obtain lower bounds on the number of measurements, purely by considering the dimension of the variety $\A$ and the number of elements from $\A$ used to construct $\bx$ (i.e., the order of the secant variety in which $\bx$ lies).  As an example the dimension of the base variety of normalized order-three tensors in $\R^{m \times m \times m}$ is $3(m-1)$.  Consequently if we were to in principle solve the tensor nuclear norm minimization problem, we should expect to require at least $\mathcal{O}(km)$ measurements to recover a rank-$k$ tensor.


\section{Computational Experiments}
\label{sec:comp}

\subsection{Algorithmic Considerations}
\label{subsec:algo}

While a variety of atomic norms can be represented or approximated by linear matrix inequalities, these representations do not necessarily translate into practical implementations.  Semidefinite programming can be technically solved in polynomial time, but general interior point solvers typically only scale to problems with a few hundred variables.  For larger scale problems, it is often preferable to exploit structure in the atomic set $\A$ to develop fast, first-order algorithms.

A starting point for first-order algorithm design lies in determining the structure of the proximity operator (or Moreau envelope) associated with the atomic norm,
\begin{equation}
	\Pi_\A(\bx;\mu): = \arg \min_\bz \tfrac{1}{2} \|\bz-\bx\|^2 + \mu\|\bz\|_{\A}\,.
\end{equation}
Here $\mu$ is some positive parameter.  Proximity operators have already been harnessed for fast algorithms involving the $\ell_1$ norm \cite{FigN2003,ComW2005,DauDD2004,HalYZ2008,WriNF2009} and the nuclear norm \cite{MaGC2008,CaiCS2008,TohY2009} where these maps can be quickly computed in closed form.  For the $\ell_1$ norm, the $i$th component of $\Pi_{\A}(\bx;\mu)$ is given by
\begin{equation}
\Pi_{\A}(\bx;\mu)_i = \begin{cases} \bx_i+\mu & \bx_i<-\mu\\
					0 & -\mu \leq \bx_i \leq \mu\\
					\bx_i-\mu & \bx_i>\mu
					\end{cases}\,.
\end{equation}
This is the so-called \emph{soft thresholding} operator.  For the nuclear norm, $\Pi_{\A}$ soft thresholds the singular values.  In either case, the only structure necessary for the cited algorithms to converge is the convexity of the norm. Indeed, essentially any algorithm developed for $\ell_1$ or nuclear norm minimization can in principle be adapted for atomic norm minimization.  One simply needs to apply the operator $\Pi_{\A}$ wherever a shrinkage operation was previously applied.

For a concrete example, suppose $f$ is a smooth function, and consider the optimization problem
\begin{equation}\label{eq:smooth-reg}
	\min_\bx~f(\bx)+\mu\|\bx\|_{\A}\,.
\end{equation}
The classical projected gradient method for this problem alternates between taking steps along the gradient of $f$ and then applying the proximity operator associated with the atomic norm.  Explicitly, the algorithm consists of the iterative procedure
\begin{equation}
	\bx_{k+1} = \Pi_{\A}( \bx_k - \alpha_k \nabla f(\bx_k); \alpha_k\lambda)
\end{equation}
where $\{\alpha_k\}$ is a sequence of positive stepsizes.  Under very mild assumptions, this iteration can be shown to converge to a stationary point of~\eqref{eq:smooth-reg}~\cite{FukM1981}.  When $f$ is convex, the returned stationary point is a globally optimal solution.  Recently, Nesterov has described a particular variant of this algorithm that is guaranteed to converge at a rate no worse than $O(k^{-1})$, where $k$ is the iteration counter~\cite{Nes2007}.  Moreover, he proposes simple enhancements of the standard iteration to achieve an $O(k^{-2})$ convergence rate for convex $f$ and a linear rate of convergence for strongly convex $f$.

If we apply the projected gradient method to the regularized inverse problem
\begin{equation}\label{eq:general-lasso}
	\min_\bx ~	\|\Phi \bx - \by\|^2 + \lambda \|\bx\|_{\A}
\end{equation}
then the algorithm reduces to the straightforward iteration
\begin{equation}
	\bx_{k+1} = \Pi_{\A}( \bx_k + \alpha_k \Phi^\dag(\by-\Phi \bx_k); \alpha_k\lambda)\,.
\end{equation}
Here \eqref{eq:general-lasso} is equivalent to~\eqref{eq:noisy-atomic-norm-primal} for an appropriately chosen $\lambda>0$ and is useful for estimation from noisy measurements.

The basic (noiseless) atomic norm minimization problem~\eqref{eq:atomic-norm-primal} can be solved by minimizing a sequence of instances of~\eqref{eq:general-lasso} with monotonically decreasing values of $\lambda$.  Each subsequent minimization is initialized from the point returned by the previous step.  Such an approach corresponds to the classic Method of Multipliers~\cite{Ber1996} and has proven effective for solving problems regularized by the $\ell_1$ norm and for total variation denoising~\cite{YinODG2007,CaiOS2008}.

This discussion demonstrates that when the proximity operator associated with some atomic set $\A$ can be easily computed, then efficient first-order algorithms are immediate.  For novel atomic norm applications, one can thus focus on algorithms and techniques to compute the associated proximity operators.   We note that, from a computational perspective,  it may be easier to compute the proximity operator via dual atomic norm.  Associated to each proximity operator is the dual operator
\begin{equation}\label{eq:dual-shrinkage}
	\Lambda_{\A}(\bx;\mu) = \arg \min_\by \tfrac{1}{2} \|\by-\bx\|^2 ~ \mathrm{s.t.}~ \|\by\|_{\A}^\ast\leq \mu
\end{equation}
By an appropriate change of variables, $\Lambda_{\A}$ is nothing more than the projection of $\mu^{-1}\bx$ onto the unit ball in the dual atomic norm:
\begin{equation}
	\Lambda_{\A}(\bx;\mu) = \arg \min_\by \tfrac{1}{2} \|\by-\mu^{-1} \bx\|^2 ~ \mathrm{s.t.}~ \|\by\|_{\A}^\ast\leq 1
\end{equation}

From convex programming duality, we have $\bx = \Pi_\A(\bx;\mu)+\Lambda_{\A}(\bx;\mu)$.  This can be seen by observing
\begin{align}
	\min_\bz  \tfrac{1}{2} \|\bz-\bx\|^2 + \mu\|\bz\|_{\A} &= 	 \min_\bz \max_{\|\by\|_{\A}^\ast\leq \mu} \tfrac{1}{2} \|\bz-\bx\|^2 + \langle \by, \bz\rangle\\
	&= \max_{\|\by\|_{\A}^\ast\leq \mu} 	\min_\bz  \tfrac{1}{2} \|\bz-\bx\|^2 + \langle \by, \bz\rangle\\
	&= \max_{\|\by\|_{\A}^\ast\leq \mu} 	-\tfrac{1}{2} \|\by-\bx\|^2 + \tfrac{1}{2}\|\bx\|^2
\end{align}
In particular, $\Pi_{\A}(\bx;\mu)$ and $\Lambda_{\A}(\bx;\mu)$ form a complementary primal-dual pair for this optimization problem.  Hence, we only need to able to efficiently compute the Euclidean projection onto the dual norm ball to compute the proximity operator associated with the atomic norm.

Finally, though the proximity operator provides an elegant framework for algorithm generation, there are many other possible algorithmic approaches that may be employed to take advantage of the particular structure of an atomic set $\A$.  For instance, we can rewrite~\eqref{eq:dual-shrinkage} as
\begin{equation}\label{eq:dual-shrinkage-explicit}
	\Lambda_{\A}(\bx;\mu) = \arg \min_\by \tfrac{1}{2} \|\by-\mu^{-1} \bx\|^2 ~~ \mathrm{s.t.}~~ \langle \by,\ba \rangle \leq 1~~\forall \ba\in \A
\end{equation}
Suppose we have access to a procedure that, given  $\bz\in\R^n$, can decide whether $\langle \bz,\ba\rangle\leq 1$ for all $\ba\in \A$, or can find a violated constraint where $\langle \bz, \hat{\ba} \rangle > 1$.  In this case, we can apply a cutting plane method or ellipsoid method to solve~\eqref{eq:dual-shrinkage} or~\eqref{eq:atomic-norm-dual}~\cite{Nes2004,Pol1997}.  Similarly, if it is simpler to compute a subgradient of the atomic norm than it is to compute a proximity operator, then the standard subgradient method~\cite{BerNO2003,Nes2004} can be applied to solve problems of the form~\eqref{eq:general-lasso}.  Each computational scheme will have different advantages and drawbacks for specific atomic sets, and relative effectiveness needs to be evaluated on a case-by-case basis.

\subsection{Simulation Results}
\label{subsec:sims}

We describe the results of numerical experiments in recovering orthogonal matrices, permutation matrices, and rank-one sign matrices (i.e., cut matrices) from random linear measurements by solving convex optimization problems.  All the atomic norm minimization problems in these experiments are solved using a combination of the SDPT3 package \cite{TohTT} and the YALMIP parser \cite{Lof2004}.

\textbf{Orthogonal matrices.} We consider the recovery of $20 \times 20$ orthogonal matrices from random Gaussian measurements via \emph{spectral norm minimization}.  Specifically we solve the convex program \eqref{eq:atomic-norm-primal}, with the atomic norm being the spectral norm.  Figure~\ref{fig:res} gives a plot of the probability of exact recovery (computed over $50$ random trials) versus the number of measurements required.

\textbf{Permutation matrices.}  We consider the recovery of $20 \times 20$ permutation matrices from random Gaussian measurements.  We solve the convex program \eqref{eq:atomic-norm-primal}, with the atomic norm being the norm induced by the Birkhoff polytope of $20 \times 20$ doubly stochastic matrices.  Figure~\ref{fig:res} gives a plot of the probability of exact recovery (computed over $50$ random trials) versus the number of measurements required.

\textbf{Cut matrices.}  We consider the recovery of $20 \times 20$ cut matrices from random Gaussian measurements.  As the cut polytope is intractable to characterize, we solve the convex program \eqref{eq:atomic-norm-primal} with the atomic norm being approximated by the norm induced by the semidefinite relaxation $\mathcal{P}_1$ described in Section~\ref{subsec:tradeoff}.  Recall that this is the second theta body associated with the convex hull of cut matrices, and so this experiment verifies that objects can be recovered from theta-body approximations.  Figure~\ref{fig:res} gives a plot of the probability of exact recovery (computed over $50$ random trials) versus the number of measurements required.

\begin{figure}
\begin{center}
\epsfig{file=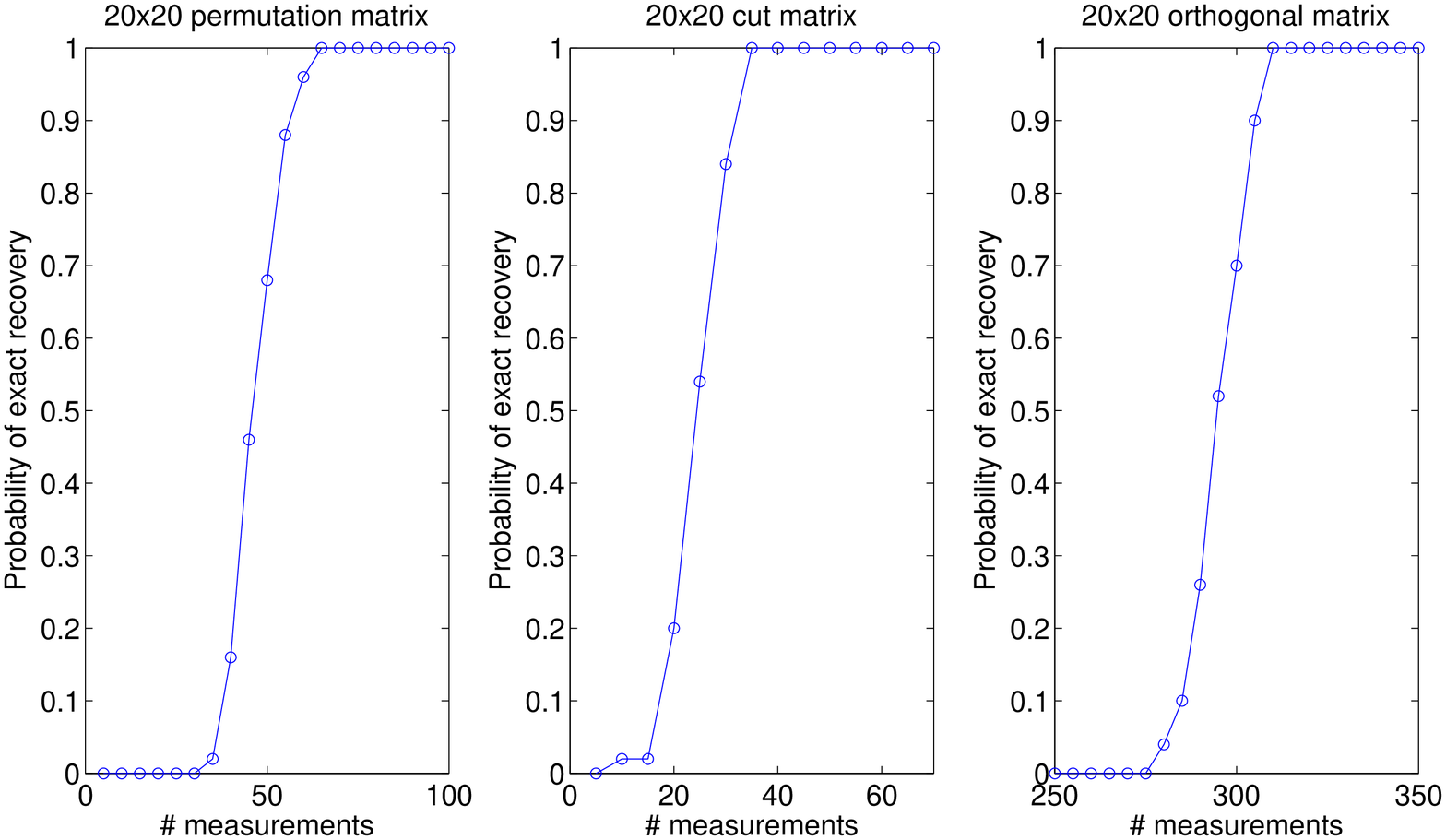,width=14cm,height=6cm} \caption{Plots of the number of measurements available versus the probability of exact recovery (computed over $50$ trials) for various models.} \label{fig:res}
\end{center}
\end{figure}

In each of these experiments we see agreement between the observed phase transitions, and the theoretical predictions (Propositions \ref{prop:ortho}, \ref{prop:birkhoff}, and \ref{prop:cut}) of the number of measurements required for exact recovery.  In particular note that the phase transition in Figure~\ref{fig:res} for the number of measurements required for recovering an orthogonal matrix is very close to the prediction $n \approx \tfrac{3m^2-m}{4} = 295$ of Proposition~\ref{prop:ortho}.  We refer the reader to \cite{DonT2005,RecFP2010,ManR2009} for similar phase transition plots for recovering sparse vectors, low-rank matrices, and signed vectors from random measurements via convex optimization.


\section{Conclusions and Future Directions}
\label{sec:conc}

This manuscript has illustrated that for a fixed set of base atoms, the atomic norm is the best choice of a convex regularizer for solving ill-posed inverse problems with the prescribed priors.  With this in mind, our results in Section~\ref{sec:gaussian} and Section~\ref{sec:rep} outline methods for computing hard limits on the number of measurements required for recovery from \emph{any} convex heuristic.  Using the calculus of Gaussian widths, such bounds can be computed in a relatively straightforward fashion, especially if one can appeal to notions of convex duality and symmetry.  This computational machinery of widths and dimension counting is surprisingly powerful: near-optimal bounds on estimating sparse vectors and low-rank matrices from partial information follow from elementary integration.  Thus we expect that our new bounds concerning symmetric, vertex-transitive polytopes are also nearly tight.  Moreover, algebraic reasoning allowed us to explore the inherent trade-offs between computational efficiency and measurement demands. More complicated algorithms for atomic norm regularization might extract structure from less information, but approximation algorithms are often sufficient for near optimal reconstructions.

This report serves as a foundation for many new exciting directions in inverse problems, and we close our discussion with a description of several natural possibilities for future work:

\paragraph{Width calculations for more atomic sets.}  The calculus of Gaussian widths described in Section~\ref{sec:gaussian} provides the building blocks for computing the Gaussian widths for the application examples discussed in Section~\ref{sec:def}.  We have not yet exhaustively estimated the widths in all of these examples, and a thorough cataloging of the measurement demands associated with different prior information would provide a more complete understanding of the fundamental limits of solving underdetermined inverse problems.  Moreover, our list of examples is by no means exhaustive.  The framework developed in this paper provides a compact and efficient methodology for constructing regularizers from very general prior information, and new regularizers can be easily created by translating grounded expert knowledge into new atomic norms.

\paragraph{Recovery bounds for structured measurements.}  Our recovery results focus on generic measurements because, for a general set $\A$, it does not make sense to delve into specific measurement ensembles.   Particular structures of the measurement matrix $\Phi$ will depend on the application and the atomic set $\A$.  For instance, in compressed sensing, much work focuses on randomly sampled Fourier coefficients~\cite{CanRT2006} and random Toeplitz and circulant matrices~\cite{HauBRN2008,Rau2009}.  With low-rank matrices, several authors have investigated reconstruction  from a small collection of entries~\cite{CanR2009}.  In all of these cases, some notion of \emph{incoherence} plays a crucial role, quantifying the amount of information garnered from each row of $\Phi$. It would be interesting to explore how to appropriately generalize notions of incoherence to new applications.  Is there a particular definition that is general enough to encompass most applications?  Or do we need a specialized concept to match the specifics of each atomic norm?

\paragraph{Quantifying the loss due to relaxation.} Section~\ref{subsec:tradeoff} illustrates how the choice of approximation of a particular atomic norm can dramatically alter the number of measurements required for recovery.  However, as was the case for vertices of the cut polytope, some relaxations incur only a very modest increase in measurement demands.  Using techniques similar to those employed in the study of semidefinite relaxations of hard combinatorial problems, is it possible to provide a more systematic method to estimate the number of measurements required to recover points from polynomial-time computable norms?

\paragraph{Atomic norm decompositions.}  While the techniques of Section~\ref{sec:gaussian} and Section~\ref{sec:rep} provide bounds on the estimation of points in low-dimensional secant varieties of atomic sets, they do not provide a procedure for actually constructing decompositions.  That is, we have provided bounds on the number of measurements required to recover points $\bx$ of the form
\begin{equation*}
	\bx=\sum_{\ba\in\A} c_\ba \ba
\end{equation*}
when the coefficient sequence $\{c_\ba\}$ is sparse, but we do not provide any methods for actually recovering $c$ itself.  These decompositions are useful, for instance, in actually computing the rank-one binary vectors optimized in semidefinite relaxations of combinatorial algorithms~\cite{GoeW1995,Nes1997,AloN2006}, or in the computation of tensor decompositions from incomplete data~\cite{KolB2009}.  Is it possible to use algebraic structure to generate deterministic or randomized algorithms for reconstructing the atoms that underlie a vector $\bx$, especially when approximate norms are used?

\paragraph{Large-scale algorithms.} Finally, we think that the most fruitful extensions of this work lie in a thorough exploration of the empirical performance and efficacy of atomic norms on large-scale inverse problems.  The proposed algorithms in Section~\ref{sec:comp} require only the knowledge of the proximity operator of an atomic norm, or a Euclidean projection operator onto the dual norm ball.  Using these design principles and the geometry of particular atomic norms should enable the scaling of atomic norm techniques to massive data sets.


\section*{Acknowledgements}

We would like to gratefully acknowledge Holger Rauhut for several suggestions on how to improve the presentation in Section 3. We would also like to thank Santosh Vempala, Joel Tropp, Bill Helton, and Jonathan Kelner for helpful discussions.


\appendix

\section{Proof of Proposition~\ref{prop:dual-width}}
\label{app:dual-width}

\begin{proof}
First note that the Gaussian width can be upper-bounded as follows:
\begin{equation}\label{eq:gwidth}
w(\mathcal{C} \cap \Sp^{p-1}) \leq \E_\bg\left[ \sup_{\bz \in \mathcal{C}\cap \mathcal{B}(0,1)} \bg^T \bz \right],
\end{equation}
where $\mathcal{B}(0,1)$ denotes the unit Euclidean ball.
The expression on the right hand side inside the expected value can be expressed as the optimal value of the following convex optimization problem for each $\bg \in \R^p$:
\begin{equation}\label{eq:primal}
\begin{array}{ll}
\max_{\bz} & \bg^T \bz\\
\mathrm{s.t.} & \bz \in \mathcal{C}\\
& \|\bz\|^2\leq 1
\end{array}
\end{equation}
We now proceed to form the dual problem of \eqref{eq:primal} by first  introducing the Lagrangian
\begin{equation*}
\mathcal{L}(\bz,\bu,\gamma) = \bg^T \bz + \gamma (1-\bz^T \bz) - \bu^T \bz
\end{equation*}
where $\bu \in \mathcal{C}^\ast$ and $\gamma \geq 0$ is a scalar.  To obtain the dual problem we maximize the Lagrangian with respect to $\bz$, which amounts to setting
\begin{equation*}
\bz = \frac{1}{2\gamma} (\bg-\bu).
\end{equation*}
Plugging this into the Lagrangian above gives the dual problem
\begin{equation*}
\begin{array}{ll}
\min & \gamma + \frac{1}{4\gamma} \|\bg-\bu\|^2\\
\mathrm{s.t.} & \bu \in \mathcal{C}^\ast\\
& \gamma \geq 0.
\end{array}
\end{equation*}
Solving this optimization problem with respect to $\gamma$ we find that $\gamma = \frac{1}{2} \|\bg-\bu\|$, which gives the dual problem to \eqref{eq:primal}
\begin{equation}\label{eq:dual}
\begin{array}{ll}
\min & \|\bg-\bu\|\\
\mathrm{s.t.} & \bu \in \mathcal{C}^\ast
\end{array}
\end{equation}
Under very mild assumptions about $\mathcal{C}$, the optimal value of \eqref{eq:dual} is equal to that of \eqref{eq:primal} (for example as long as $\mathcal{C}$ has a non-empty relative interior, strong duality holds).  Hence we have derived
\begin{equation}\label{eq:dual-width}
\E_\bg \left[ \sup_{\bz\in \mathcal{C}\cap \mathcal{B}(0,1)} \bg^T \bz\right] =\E_\bg \left[\mathrm{dist}(\bg,\mathcal{C}^\ast)\right].
\end{equation}
This equation combined with the bound \eqref{eq:gwidth} gives us the desired result.
\end{proof}

\section{Proof of Theorem~\ref{theo:angle}}
\label{app:width-angle}

\begin{proof}
We set $\beta = \tfrac{1}{\Theta}$.  First note that if $\beta \geq \tfrac{1}{4} \exp\{\tfrac{p}{9}\}$ then the width bound exceeds $\sqrt{p}$, which is the maximal possible value for the width of $\mathcal{C}$.  Thus, we will assume throughout that  $\beta \leq \tfrac{1}{4} \exp\{\tfrac{p}{9}\}$.

Using Proposition~\ref{prop:dual-width} we need to upper bound the expected distance to the polar cone.  Let $\bg \sim \mathcal{N}(0,I)$ be a normally distributed random vector.  Then the norm of $\bg$ is independent from the angle of $\bg$.  That is, $\|\bg\|$ is independent from $\bg/\|\bg\|$.   Moreover, $\bg/\|\bg\|$ is distributed as a uniform sample on $\Sp^{p-1}$, and $\E_\bg[\|\bg\|]\leq \sqrt{p}$.  Thus we have
\begin{equation}
	\E_\bg[\mathrm{dist}(\bg,\mathcal{C}^\ast)] \leq \E_\bg[\|\bg\|\cdot \mathrm{dist}(\bg/\|\bg\|, \mathcal{C}^\ast \cap \Sp^{p-1})]\leq \sqrt{p}\E_\bu[ \mathrm{dist}(\bu, \mathcal{C}^\ast \cap \Sp^{p-1})]
\end{equation}
where $\bu$ is sampled uniformly on $\Sp^{p-1}$.

To bound the latter quantity, we will use isoperimetry.  Suppose $A$ is a subset of $\Sp^{p-1}$ and $B$ is a spherical cap with the same volume as $A$.  Let $N(A,r)$ denote the locus of all points in the sphere of Euclidean distance at most $r$ from the set $A$.  Let $\mu$ denote the Haar measure on $\Sp^{p-1}$ and $\mu(A;r)$ denote the measure of $N(A,r)$.  Then spherical isoperimetry states that $\mu(A;r)\geq \mu(B;r)$ for all $r\geq 0$ (see, for example \cite{Led2000,Mat2002}).

Let $B$ now denote a spherical cap with $\mu(B)=\mu(\mathcal{C}^\ast\cap \Sp^{p-1})$.  Then we have
\begin{align}
	\E_\bu[ \mathrm{dist}(\bu, \mathcal{C}^\ast \cap \Sp^{p-1})] &= \int_0^\infty \P [ \mathrm{dist}(\bu, \mathcal{C}^\ast \cap \Sp^{p-1})>t]dt\\
	&= \int_0^\infty (1-\mu(\mathcal{C}^\ast\cap \Sp^{p-1};t))dt\\
	&\leq \int_0^\infty (1-\mu(B;t))dt
\end{align}
where the first equality is the integral form of the expected value and the last inequality follows by isoperimetry. Hence we can bound the expected distance to the polar cone intersecting the sphere using only knowledge of the volume of spherical caps on $\Sp^{p-1}$.

To proceed let $v(\varphi)$ denote the volume of a spherical cap subtending a solid angle $\varphi$. An explicit formula for $v(\varphi)$ is
\begin{equation}
	v(\varphi)= z_p^{-1}\int_0^\varphi \sin^{p-1}(\vartheta)d\vartheta
\end{equation}
where $z_p = \int_0^\pi \sin^{p-1}(\vartheta)d\vartheta$ \cite{KlaR1997}.   Let $\varphi(\beta)$ denote the minimal solid angle of a cap such that $\beta$ copies of that cap cover $\Sp^{p-1}$.  Since the geodesic distance on the sphere is always greater than or equal to Euclidean distance, if $K$ is a spherical cap subtending $\psi$ radians, $\mu(K;t)\geq v(\psi+t)$.   Therefore
\begin{equation}\label{eq:cap-bound}
\int_0^\infty (1-\mu(B;t))dt \leq  \int_0^\infty (1-v(\varphi(\beta)+t))dt \,.
\end{equation}
We can proceed to simplify the right-hand-side integral:
\begin{align}
\int_0^\infty (1-v(\varphi(\beta)+t))dt 	&= \int_0^{\pi-\varphi(\beta)} (1-v(\varphi(\beta)+t))dt \\
	&= \pi-\varphi(\beta) - \int_0^{\pi-\varphi(\beta)} v(\varphi(\beta)+t)dt \\
	&= \pi-\varphi(\beta) - z_p^{-1} \int_0^{\pi-\varphi(\beta)} \int_0^{\varphi(\beta)+t}  \sin^{p-1}\vartheta d\vartheta dt \\
\label{eq:int-order} &= \pi-\varphi(\beta) - z_p^{-1} \int_0^{\pi} \int_{\max(\vartheta-\varphi(\beta),0)}^{\pi-\varphi(\beta)} \sin^{p-1}\vartheta dt d\vartheta  \\
		&= \pi-\varphi(\beta) - z_p^{-1} \int_0^{\pi} \left\{\pi-\varphi(\beta)-\max(\vartheta-\varphi(\beta),0)\right\}\sin^{p-1}\vartheta d\vartheta  \\
		&= z_p^{-1} \int_0^{\pi} \max(\vartheta-\varphi(\beta),0)\sin^{p-1}\vartheta d\vartheta  \\
		&= z_p^{-1} \int_{\varphi(\beta)}^{\pi} (\vartheta-\varphi(\beta))\sin^{p-1}\vartheta d\vartheta
\end{align}
\eqref{eq:int-order} follows by switching the order of integration and the rest of these equalities follow by straight-forward integration and some algebra.

Using the inequalities that $z_p \geq \frac{2}{\sqrt{p-1}}$ (see \cite{Led2000}) and $\sin(x)\leq \exp(-(x-\pi/2)^2/2)$ for $x\in[0,\pi]$, we can bound the last integral as
\begin{align}
	 z_p^{-1} \int_{\varphi(\beta)}^{\pi} (\vartheta-\varphi(\beta))\sin^{p-1}\vartheta d\vartheta &\leq \frac{\sqrt{p-1}}{2} \int_{\varphi(\beta)}^{\pi} (\vartheta-\varphi(\beta)) \exp\left(-\frac{p-1}{2}(\vartheta-\tfrac{\pi}{2})^2\right)d\vartheta
\end{align}
Performing the change of variables $a = \sqrt{p-1}(\vartheta-\tfrac{\pi}{2})$, we are left with the integral
\begin{align}
& \frac{1}{2} \int_{\sqrt{p-1}(\varphi(\beta)-\pi/2)}^{\sqrt{p-1} \pi/2} \left\{\frac{a}{\sqrt{p-1}}+\left(\frac{\pi}{2}-\varphi(\beta)\right)\right\} \exp\left(-\frac{a^2}{2}\right)da\\
=& -\frac{1}{2\sqrt{p-1}} \exp\left(-\frac{a^2}{2}\right) \bigg|_{\sqrt{p-1}(\varphi(\beta)-\pi/2)}^{\sqrt{p-1} \pi/2}+
 \frac{ \frac{\pi}{2}-\varphi(\beta)}{2} \int_{\sqrt{p-1}(\varphi(\beta)-\pi/2)}^{\sqrt{p-1} \pi/2} \exp\left(-\frac{a^2}{2}\right)da\\
 \label{eq:final-distance-bound} \leq& \frac{1}{2\sqrt{p-1}} \exp\left(-\frac{p-1}{2}(\pi/2-\varphi(\beta))^2\right) + \sqrt{\frac{\pi}{2}}\left( \frac{\pi}{2}-\varphi(\beta)\right)
 \end{align}
In this final bound, we bounded the first term by dropping the upper integrand, and for the second term we used the fact that
\begin{equation}
\int_{-\infty}^\infty \exp(-x^2/2) dx = \sqrt{2\pi}\,.
\end{equation}

We are now left with the task of computing a lower bound for $\varphi(\beta)$.   We need to first reparameterize the problem.  Let $K$ be a spherical cap.  Without loss of generality, we may assume that
\begin{equation}
	K = \{ x\in\Sp^{p-1}~:~x_1\geq h\}
\end{equation}
for some $h \in [0,1]$.  $h$ is the \emph{height} of the cap over the equator.  Via elementary trigonometry, the solid angle that $K$ subtends is given by $\pi/2-\sin^{-1}(h)$.  Hence, if $h(\beta)$ is the largest number such that $\beta$ caps of height $h$ cover $\Sp^{p-1}$, then $h(\beta)=\sin(\pi/2-\varphi(\beta))$.

The quantity $h(\beta)$ may be estimated using the following estimate from~\cite{Bri1998}.  For $h\in[0,1]$, let $\gamma(p,h)$ denote the volume of a spherical cap of $\Sp^{p-1}$  of height $h$.
\begin{lemma}[\cite{Bri1998}]
For $1\geq h\geq \frac{2}{\sqrt{p}}$,
\begin{equation}
\frac{1}{10 h \sqrt{p}}(1-h^2)^{\frac{p-1}{2}} \leq \gamma(p,h) \leq
\frac{1}{2 h \sqrt{p}}(1-h^2)^{\frac{p-1}{2}} \,.
\end{equation}
\end{lemma}

Note that for $h \geq \frac{2}{\sqrt{p}}$,
\begin{equation}
	\frac{1}{2 h \sqrt{p}}(1-h^2)^{\frac{p-1}{2}}   \leq 	 \frac{1}{4}(1-h^2)^{\frac{p-1}{2}}   \leq \frac{1}{4}\exp(-\tfrac{p-1}{2} h^2)\,.
\end{equation}
So if
\begin{equation}
	h = \sqrt{\frac{2\log(4\beta)}{p-1}}
\end{equation}
then $h\leq 1$ because we have assumed $\beta \leq \tfrac{1}{4} \exp\{\tfrac{p}{9}\}$ and $p \geq 9$.  Moreover,
$h\geq\frac{2}{\sqrt{p}}$ and the volume of the cap with height $h$ is less than or equal to $1/\beta$.  That is
\begin{equation}
\varphi(\beta)\geq \pi/2 - \sin^{-1}\left( \sqrt{\frac{2\log(4\beta)}{p-1}}\right)\,.
\end{equation}
Combining the estimate \eqref{eq:final-distance-bound} with Proposition~\ref{prop:dual-width}, and using our estimate for $\varphi(\beta)$, we get the bound
\begin{equation}\label{eq:complicated-looking-bound}
w(\mathcal{C}) \leq \frac{1}{2}\sqrt{\frac{p}{p-1}}\exp\left(-\tfrac{p-1}{2} \sin^{-1}\left( \sqrt{\frac{2\log(4\beta)}{p-1}}\right)^2\right) + \sqrt{\frac{\pi p}{2}}\sin^{-1}\left( \sqrt{\frac{2\log(4\beta)}{p-1}}\right)
\end{equation}
This expression can be simplified by using the following bounds. First, $\sin^{-1}(x)\geq x$ lets us upper bound the first term by $\sqrt{\frac{p}{p-1}}\frac{1}{8\beta}$. For the second term, using the inequality $\sin^{-1}(x)\leq \tfrac{\pi}{2}x$ results in the upper bound
\begin{equation}
w(\mathcal{C}) \leq \sqrt{\frac{p}{p-1}}\left( \frac{1}{8\beta} + \frac{\pi^{3/2}}{2}  \sqrt{\log(4\beta)}\right).
\end{equation}
For $p\geq 9$ the upper bound can be expressed simply as $w(\mathcal{C})\leq 3\sqrt{\log(4 \beta)}$.  We recall that $\beta = \tfrac{1}{\Theta}$, which completes the proof of the theorem.
\end{proof}

\section{Direct Width Calculations}
\label{app:direct}

We first give the proof of Proposition~\ref{prop:l1}.
\begin{proof}

Let $\bxs$ be an $s$-sparse vector in $\R^p$ with $\ell_1$ norm equal to $1$, and let $\A$ denote the set of unit-Euclidean-norm one-sparse vectors.  Let $\Delta$ denote the set of coordinates where $\bxs$ is non-zero. The normal cone at $\bxs$ with respect to the $\ell_1$ ball is given by
\begin{align}
	N_\A(\bxs) &= \mathrm{cone}\left\{\bz\in\R^p~:~\bz_i= \mathrm{sgn}(\bxs_i)~\mbox{for}~i\in \Delta,\,\,\, |\bz_i|\leq1~\mbox{for}~i\in \Delta^c\right\}\\
	&= \left\{\bz\in\R^p~:~\bz_i= t\mathrm{sgn}(\bxs_{i})~\mbox{for}~i\in \Delta,\,\,\, |\bz_i|\leq t~\mbox{for}~i\in \Delta^c~\mbox{for some}~t>0\right\}\,.
\end{align}
Here $\Delta^c$ represents the zero entries of $\bxs$. The minimum squared distance to the normal cone at $\bxs$ can be formulated as a one-dimensional convex optimization problem for arbitrary $\bz\in\R^p$
\begin{align}
\inf_{\bu\in N_\A(\bxs)} \|\bz-\bu\|_2^2 &=
\inf_{\stackrel{t\geq 0}{|\bu_i|<t,\,\,i\in \Delta^c}} \sum_{i\in \Delta} (\bz_i-t\mathrm{sgn}(\bxs_i))^2 + \sum_{j\in \Delta^c}  (\bz_j - \bu_j)^2\\
&=\inf_{t\geq 0} \sum_{i\in \Delta} (\bz_i-t\mathrm{sgn}(\bxs_{i}))^2 + \sum_{j\in \Delta^c}  \mathrm{shrink}(\bz_j,t)^2
\end{align}
where
\begin{equation}\label{eq:shrink}
	\mathrm{shrink}(z,t) = \begin{cases} z+t & z<-t\\
	0 &-t\leq z \leq t\\
	z - t & z>t
	 \end{cases}
\end{equation}
is the $\ell_1$-shrinkage function. Hence, for any fixed $t\geq 0$ independent of $\bg$, we have
\begin{align}
	\E\left[\inf_{\bu \in N_\A(\bxs) } \|\bg-\bu\|_2^2\right]
	&\leq  \E\left[\sum_{i\in \Delta} ( \bg_i-t\mathrm{sgn}(\bxs_i) )^2
		+ \sum_{j\in \Delta^c}  \mathrm{shrink}(\bg_j,t)^2\right]\\
\label{eq:ell-1-width-bound} &= s(1+t^2)
	+ \E\left[\sum_{j\in \Delta^c}  \mathrm{shrink}(\bg_j,t)^2\right]\,.
\end{align}

Now we directly integrate the second term, treating each summand individually.  For a zero-mean, unit-variance normal random variable $g$,
\begin{align}
\E\left[ \mathrm{shrink}(g,t)^2\right] &= \frac{1}{\sqrt{2\pi}} \int_{-\infty}^{-t} (g+t)^2\exp(-g^2/2) dg +  \frac{1}{\sqrt{2\pi}}\int_t^\infty (g-t)^2 \exp(-g^2/2)dg\\
&= \frac{2}{\sqrt{2\pi}} \int_{t}^{\infty} (g-t)^2\exp(-g^2/2) dg \\
&= -\frac{2}{\sqrt{2\pi}} t\exp(-t^2/2) +
 \frac{2(1+t^2)}{\sqrt{2\pi}} \int_{t}^{\infty} \exp(-g^2/2) dg\\
 &\leq \frac{2}{\sqrt{2\pi}}\left(-t  +\frac{1+t^2}{t} \right)\exp(-t^2/2)\\
 &=\frac{2}{\sqrt{2\pi}}\frac{1}{t}\exp(-t^2/2)\,.
\end{align}
The first simplification follows because the $\mathrm{shrink}$ function and Gaussian distributions are symmetric about the origin.  The second equality follows by integrating by parts. The inequality follows by a tight bound on the Gaussian $Q$-function
\begin{equation}
 Q(x) =  \frac{1}{\sqrt{2\pi}} \int_{x}^{\infty} \exp(-g^2/2) dg \leq \frac{1}{\sqrt{2\pi}}\frac{1}{x}\exp(-x^2/2)\quad\mbox{for }x>0\,.
\end{equation}
Using this bound, we get
\begin{equation}\label{eq:q-fun-bnd}
\E\left[\inf_{\bu\in N_\A(\bxs)} \|\bg-\bu\|_2^2\right]  \leq s(1+t^2) +(p-s)\frac{2}{\sqrt{2\pi}}\frac{1}{t}\exp(-t^2/2)
\end{equation}
Setting $t= \sqrt{2\log(p/s)}$ gives
\begin{equation}\label{eq:l1-precise-bound}
\E\left[\inf_{\bz\in N_\A(\bxs)} \|\bg-\bz\|_2^2\right]  \leq s\left(1+2\log\left(\frac{p}{s}\right)\right)+\frac{s(1-s/p)}{\pi \sqrt{\log(p/s)}}\leq
2s\log\left(p/s\right)+\tfrac{5}{4}s\,.
\end{equation}
The last inequality follows because
\begin{equation}
\frac{(1-s/p)}{\pi \sqrt{\log(p/s)}}\leq 0.204<1/4
\end{equation}
whenever $0 \leq s\leq p$.

\end{proof}

Next we give the proof of Proposition~\ref{prop:nuclear}.

\begin{proof}

Let $\bxs$ be an $m_1 \times m_2$ matrix of rank $r$ with singular value decomposition $U\Sigma V^*$, and let $\A$ denote the set of rank-one unit-Euclidean-norm matrices of size $m_1 \times m_2$.   Without loss of generality, impose the conventions $m_1\leq m_2$, $\Sigma$ is $r\times r$, $U$ is $m_1 \times r$, $V$ is $m_2 \times r$, and assume the nuclear norm of $\bxs$ is equal to $1$.

Let $\bu_k$ (respectively $\bv_k$) denote the $k$'th column of $U$ (respectively $V$).   It is convenient to introduce the orthogonal decomposition $\R^{m_1 \times m_2} = \Delta \oplus \Delta^\perp$ where $\Delta$ is the linear space spanned by elements of the form $\bu_k \bz^T$ and $\by \bv_k^T$, $1 \le k \le r$, where $\bz$ and $\by$ are arbitrary, and $\Delta^\perp$ is the orthogonal complement of $\Delta$.  The space $\Delta^\perp$ is the subspace of matrices spanned by the family $(\by \bz^T)$, where $\by$ (respectively $\bz$) is any vector orthogonal to all the columns of $U$ (respectively $V$). The normal cone of the nuclear norm ball at $\bxs$ is given by the cone generated by the subdifferential at $\bxs$:
\begin{align}
	N_\A(\bxs) &= \mathrm{cone}\left\{UV^T+W\in\R^{m_1\times m_2}~:~W^TU = 0,\,\,\,\,WV=0,\,\,\,\,\|W\|_\A^\ast\leq 1\right\}\\
	&=\left\{tUV^*+W\in\R^{m_1\times m_2}~:~W^TU = 0,\,\,\,\,WV=0,\,\,\,\,\|W\|_\A^\ast \leq t,\,\,\,t\geq 0\right\}\,.
\end{align}
Note that here $\|Z\|_\A^\ast$ is the operator norm, equal to the maximum singular value of $Z$~\cite{RecFP2010}.

Let $G$ be a Gaussian random matrix with i.i.d. entries, each with mean zero and unit variance.  Then the matrix
\begin{equation}
	Z(G) = \|\PTc(G)\| UV^* + \PTc(G)
\end{equation}
is in the normal cone at $\bxs$.  We can then compute
\begin{align}
	\E\left[\|G-Z(G)\|_F^2\right] & = \E\left[\|\PT(G)-\PT(Z(G))\|_F^2\right] \\
\label{eq:matrix-indep} &= \E\left[\|\PT(G)\|_F^2\right]+\E\left[\|\PT(Z(G))\|_F^2\right] \\
	&=  r(m_1+m_2-r) + r\E[\|\PTc(G)\|^2]\,. \label{eq:nuclear-width-bound-direct}
\end{align}
Here \eqref{eq:matrix-indep} follows because $\PT(G)$ and $\PTc(G)$ are independent.  The final line follows because $\mathrm{dim}(T)=r(m_1+m_2-r)$ and the Frobenius (i.e., Euclidean) norm of $UV^*$ is $\|UV^*\|_F=\sqrt{r}$.  Due to the isotropy of Gaussian random matrices,  $\PTc(G)$ is identically distributed as an $(m_1-r)\times(m_2-r)$ matrix with i.i.d. Gaussian entries each with mean zero and variance one.  We thus know that
\begin{equation}
	\P\left[ \|\PTc(G)\|\geq \sqrt{m_1-r}+\sqrt{m_2-r}  +s\right] \leq \exp\left(-s^2/2\right)	
\end{equation}
(see, for example,~\cite{DavS2001}).
To bound the latter expectation, we again use the integral form of the expected value.  Letting $\mu_{T^\perp}$ denote the quantity $\sqrt{m_1-r}+\sqrt{m_2-r}$, we have
\begin{align}
\E\left[\|\PTc(G)\|^2\right] &= \int_0^\infty \P\left[\|\PTc(G)\|^2>h\right] dh\\
&\leq \mu_{T^\perp}^2 +  \int_{ \mu_{T^\perp}^2}^\infty\P\left[\|\PTc(G)\|^2>h\right] dh\\
&\leq  \mu_{T^\perp}^2 +  \int_0^\infty\P\left[\|\PTc(G)\|^2>  \mu_{T^\perp}^2 + t \right] dt\\
&\leq  \mu_{T^\perp}^2 +  \int_0^\infty\P\left[\|\PTc(G)\|> \mu_{T^\perp} +\sqrt{t} \right] dt\\
&\leq \mu_{T^\perp}^2 +  \int_0^\infty \exp(-t/2) dt\\
&=  \mu_{T^\perp}^2 +  2
\end{align}

Using this bound in~\eqref{eq:nuclear-width-bound-direct}, we get that
\begin{align}\label{eq:nuclear-crude-bound}
\E\left[\inf_{Z\in N_\A(\bxs)} \|G-Z\|_F^2\right] & \leq r(m_1+m_2-r) + r(\sqrt{m_1-r}+\sqrt{m_2-r})^2+ 2r\\
& \leq r(m_1+m_2-r) + 2r(m_1+m_2-2r)+ 2r\\
& \leq 3r(m_1+m_2-r)
\end{align}
where the second inequality follows from the fact that $(a+b)^2\leq 2a^2+2b^2$.  We conclude that $3r(m_1+m_2-r)$ random measurements are sufficient to recover a rank $r$, $m_1\times m_2$ matrix using the nuclear norm heuristic.
\end{proof}


\end{document}